\documentclass[12pt]{article}
\usepackage{fullpage}

\usepackage{amsmath,amsfonts,amsthm,amssymb,subfiles}

\usepackage{graphicx, multicol}
\usepackage[colorlinks=true, allcolors=blue]{hyperref}

\newtheorem{lemma}{Lemma}[section]

\newtheorem{theorem}[lemma]{Theorem}
\newtheorem{conjecture}[lemma]{Conjecture}
\newtheorem{corollary}[lemma]{Corollary}

\theoremstyle{define}

\newtheorem{remark}[lemma]{Remark}
\newtheorem{definition}[lemma]{Definition}

\makeatletter
\newtheorem*{rep@theorem}{\rep@title}
\newcommand{\newreptheorem}[2]{%
\newenvironment{rep#1}[1]{%
 \def\rep@title{#2 \ref{##1}}%
 \begin{rep@theorem}}%
 {\end{rep@theorem}}}
\makeatother

\newreptheorem{theorem}{Theorem}
\newreptheorem{lemma}{Lemma}

\newcommand{\maxroot}[1]{\operatorname{maxroot} \left(#1 \right)}

\newcommand{\cauchy}[2]{\mathcal{G}_{#1} \left(#2 \right)}

\newcommand{\deriv}[1]{D}

\newcommand{\rrpos}[1]{\mathcal{P}^{+}_{#1}}
\newcommand{\rrposle}[1]{\mathcal{P}^{+}_{\leq #1}}

\renewcommand{\hat}{\widehat}
\renewcommand{\tilde}{\widetilde}

\newcommand{\phat}{\hat{p}}
\newcommand{\ptil}{\tilde{p}}

\renewcommand{\aa}{\lambda}
\newcommand{\bb}{\mu}
\newcommand{\uu}{\alpha}

\newcommand{\AND}{\quad\text{and}\quad}

\newcommand{\Utrans}[1]{U_{\uu} {#1}}
\newcommand{\Vtrans}[2]{\mathbb{V}^{#1}{#2}}
\newcommand{\Strans}[1]{\mathbb{S}#1}

\newcommand{\Htrans}[2]{\mathcal{H}^{#1}_{#2}}
\newcommand{\Wtrans}[3]{W^{#1}_{#2} {#3}}

\newcommand{\R}{\mathbb{R}}
\newcommand{\C}{\mathbb{C}}

\renewcommand{\d}[1]{\mathrm{d}{#1}}

\newcommand{\geg}[2]{C_{#1}^{(#2)}}

\newcommand{\ii}{i}
\newcommand{\ident}{\mathbb{I}}
\newcommand{\mydet}[1]{\mathrm{det}\left[ {#1} \right]}

\newcommand{\D}[1]{\Delta[#1]}
\newcommand{\Dp}{\D{p}}

\newcommand{\fit}[1]{{#1}-orthogonal-fit}

\newcommand{\myphi}[6]{\phi^{#2, #3}_{#4}(#5, #6)}

\newcommand{\mysum}[2]{\boxplus_{#1}^{#2}}

\newcommand{\jj}{j} 
\newcommand{\nn}{n} 

\newcommand{\dd}{d} 
\newcommand{\kk}{n} 

\newcommand{\one}{\mathbf{1}}

\newcommand{\arxiv}[1]{\href{https://arxiv.org/abs/#1}{\texttt{arXiv:#1}}}

\title{A rectangular additive convolution for polynomials}
\date{\today}
\author{Aurelien Gribinski\\
Princeton University
\and Adam W. Marcus\thanks{Research done under the support of NSF CAREER grant 
DMS-1552520 and a Von Neumman Fellowship at the Institute of Advanced Study, 
NSF grant DMS-1128155.}\\
Princeton University
}

\begin{document}

\maketitle

\begin{abstract}
Motivated by the study of singular values of random rectangular matrices, we define and study the rectangular additive convolution of polynomials with nonnegative real roots.
Our definition directly generalizes the asymmetric additive convolution 
introduced by Marcus, Spielman and Srivastava (2015), and our main theorem 
gives the 
corresponding generalization of the bound on the largest root from that paper.
The main tool used in the analysis is a differential operator derived from the 
``rectangular Cauchy transform'' introduced by Benaych-Georges (2009).
The proof is inductive, with the base case requiring a new nonasymptotic bound 
on the Cauchy transform of Gegenbauer polynomials which may be of independent 
interest.
\end{abstract}
\section{Introduction}\label{sec:intro}

This paper introduces the rectangular additive convolution to the theory of finite free probability.
The motivation for a finite analogue of free probability came from a series of 
works that used expected characteristic polynomials to study certain 
combinatorial problems in linear algebra \cite{convolutions4, IF3, IF4}.
It is well known that the characteristic polynomial of a Hermitian matrix $A$ 
is the real-rooted polynomial $\chi_A(x) = \mydet{x \ident- A}$, and in 
\cite{IF3} 
the authors showed that one could analyze the effect of certain random 
perturbations on $A$ by studying the effect of certain differential operators 
applied to $\chi_A(x)$.  
In particular, they showed that any valid bounds on the largest root of the 
transformed polynomials could be translated into statements about the randomly 
perturbed matrix.
In \cite{convolutions4}, the same authors showed that these differential 
operators were actually a special case of a more general class of ``polynomial 
convolutions" and they introduced a technique for deriving bounds on the 
largest root under certain convolutions.
One interesting property of the bounds that came from this technique was that 
they had quite similar form to known identities from free probability.
This idea was strengthened by the realization that the major tools used in 
proving these bounds had striking similarities to tools in free probability.

The connection was formalized in \cite{ffmain}, where it was shown that the inequalities derived for two of the convolutions studied in \cite{convolutions4} --- the symmetric additive and multiplicative convolutions --- converge to the $R$- and $S$-transform identities of free probability (respectively).
Since the release of \cite{ffmain}, a number of advances have been made in understanding the relationship between free probability and polynomial convolutions, most notably the work of Arizmendi and Perales \cite{cumulants} in developing a combinatorial framework for finite free probability using finite free cumulants (the approach in \cite{ffmain} is primarily analytic).

The purpose of this paper is to introduce a convolution on polynomials that generalizes the third 
convolution studied in \cite{convolutions4} --- what is there called the {\em asymmetric additive convolution} --- and to prove the corresponding bound on the largest root.
The method of proof will be similar to the one used in \cite{convolutions4}, however there will be a number of added complications.
Those familiar with \cite{convolutions4} may recall that all of the inequalities proved there utilized various levels of induction to reduce to a small set of ``base cases.''
One of the difficulties in dealing with the asymmetric additive convolution (as opposed to the other two convolutions) was the fact that the corresponding base case was highly nontrivial.
Rather, it required a bound on the Cauchy transform of Chebyshev polynomials 
that was both unknown at the time and not particularly easy to derive.

We will encounter the same issue: replacing the analysis on Chebyshev polynomials will be an analysis on (the more general) Gegenbauer polynomials.
To establish the bound, we will prove a number of inequalities relating nonasymptotic properties of Gegenbauer polynomials (which appear to be unknown) with the corresponding asymptotic properties (many of which are known) and these could be of interest in their own right (see Section~\ref{sec:structure} for the location of results).

Many of the ideas required in generalizing the various constructs in \cite{convolutions4} to the ones used in this paper were inspired by the work of Florent Benaych-Georges, in particular \cite{BG} where the appropriate transforms for calculating the rectangular additive convolution of two freely independent rectangular operators were introduced (hence the name of our convolution).
We remark on this connection briefly in Section~\ref{sec:BG}, but in general have written the paper in a way that assumes no previous knowledge of free probability.

\subsection{Previous work}\label{sec:previous}

The primary predecessors of this work are \cite{convolutions4}, where other polynomial convolutions were introduced, and \cite{BG}, where the free probability version of the rectangular additive convolution was introduced (see \ref{sec:BG} for more discussion on the relation to \cite{BG}).
The original purpose of \cite{convolutions4} was to develop a generic way to bound the largest root of a real-rooted polynomial when certain differential operators were applied.
Such bounds are useful in tandem with the ``method of interlacing polynomials'' first introduced in \cite{IF1}.
One of the main inequalities in \cite{convolutions4} is the following:
\begin{theorem}[Marcus, Spielman, Srivastava]\label{thm:symmetric}
Let $p$ and $q$ be real rooted polynomials of degree at most $d$.
Then 
\[
\maxroot{U_\alpha[p \mysum{\dd}{} q]} \leq \maxroot{U_\alpha p} + \maxroot{U_\alpha q} - \alpha \dd
\]
\end{theorem}
The operator $\mysum{\dd}{}$ here is what is called the {\em symmetric additive convolution} in \cite{convolutions4} and $U_\alpha$ is the differential operator $1 - \alpha \partial$.
The $U_\alpha$ acts to smoothen the roots of the polynomials, making the convolution more predictable.
Theorem~\ref{thm:symmetric} is used in \cite{IF3} to prove an asymptotically tight version of {\em restricted invertibility}, a theorem first introduced by Bourgain and Tzafriri that has seen a wide variety of uses in mathematics (see \cite{assaf}).

A considerably more difficult inequality in \cite{convolutions4} concerned what the authors called the {\em asymmetric additive convolution}.
In the notation of this paper --- see (\ref{eq:defrac}) and (\ref{eq:SandV}) --- this inequality reads
\begin{theorem}[Marcus, Spielman, Srivastava]\label{thm:asymmetric}
Let $p$ and $q$ be polynomials of degree at most $\dd$ with nonnegative real roots.
Then 
\[
\maxroot{U_\alpha \Strans{[p \mysum{\dd}{0} q]}} \leq \maxroot{U_\alpha \Strans{p}} + \maxroot{U_\alpha{\Strans{q}}} - 2\alpha \dd.
\]
for all $\alpha \leq 0$.
Furthermore, equality holds if and only if $p(x) = x^{\dd}$ or $q(x) = x^{\dd}$. 
\end{theorem}
Theorem~\ref{thm:asymmetric} was then used in \cite{IF4} to prove the existence of bipartite Ramanujan graphs of all degrees and all sizes.
The main theorem of this paper (Theorem~\ref{thm:main}) is a generalization of Theorem~\ref{thm:asymmetric}.
In particular, we show that the generalized convolution $\mysum{\dd}{\kk}$ (defined in Section~\ref{sec:convolution}) satisfies a similar inequality:
\begin{theorem}\label{thm:main}
Let $p$ and $q$ be polynomials of degree at most $\dd$ with nonnegative real roots.
Then 
\begin{equation}\label{eq:theta}
\Theta^{\nn}_{\uu}(p \mysum{\dd}{\kk} q) \leq \Theta^{\nn}_{\uu}(p) + \Theta^{\nn}_{\uu}(q) - (\kk + 2\dd)\uu
\end{equation}
for all $\uu > 0$ and $\nn \geq 0$, where
\[
\Theta^{\nn}_\uu(p) := \sqrt{\nn^2 \uu^2 + [\maxroot{\Wtrans{\nn}{\uu}{p}}]^2}.
\]
Furthermore, equality holds if and only if $p(x) = x^{\dd}$ or $q(x) = x^{\dd}$. 
\end{theorem}

One obvious difference between the two theorems is the replacement of the $U_\alpha$ operator with the more general $\Wtrans{\nn}{\uu}{}$ operator, defined in Section~\ref{sec:measure} (the differences between the two are discussed in Remark~\ref{rem:WvsU}).
Much of the added difficulty in the proof of Theorem~\ref{thm:main} (with 
respect to the proof of Theorem~\ref{thm:asymmetric}) is the quadratic nature 
of the $\Wtrans{\nn}{\uu}{}$ operator.

The other obvious difference is that the bound in \eqref{eq:theta} no longer 
relates the $\maxroot{}$ of the convolution to that of the original polynomials 
in a linear way (unless $\nn = 0$)
We have yet to come up with an intuitive explanation of why this is the correct 
form, apart from it coming up naturally in the work of Benaych-Georges (see 
Section~\ref{sec:BG}).
However, it was worth noting that the form of $\Wtrans{\nn}{\uu}{}$ (and 
$\Theta^{\nn}_{\uu}$) only plays a significant role in Section~\ref{sec:proof} 
(the base cases).  
The inductive steps in Section~\ref{sec:inductions} can be adapted to work with 
a much larger family of operators (and quantities) with the appropriate 
monotonicity properties (as was pointed out by an anonymous referee).
One corollary of the results in Section~\ref{sec:measure}, however is that, for 
fixed $\nn$, the quantity $\Theta^{\nn}_\uu(p)$ is increasing in $\uu$.
Given that $\Theta^{\nn}_0(p) = \maxroot{p}$, this will provide the type of bound we are hoping. quantitatively bound the effects of $\mysum{\dd}{\kk}$ on the largest root (with respect to the input polynomials).

In a different direction, Leake and Ryder showed that 
Theorem~\ref{thm:symmetric} was actually a special case of a more general 
submodularity inequality \cite{submodular} (note that they use the notation 
$\mysum{}{\dd}$ as opposed to the $\mysum{\dd}{}$ notation introduced in 
\cite{convolutions4}).
\begin{theorem}[Leake, Ryder]\label{thm:submodular}
For any real rooted polynomials $p, q, r$ of degree at most $d$, one has 
\[
\maxroot{[p \mysum{\dd}{} q \mysum{\dd}{} r]} + \maxroot{r} \leq \maxroot{p \mysum{\dd}{} r } + \maxroot{q \mysum{\dd}{} r }
\]
\end{theorem}
Theorem~\ref{thm:symmetric} is then the case $r(x) = x^\dd - \frac{\alpha}{\dd} x^{\dd-1}$.
The proof method of \cite{submodular} is similar to that of \cite{convolutions4} (and therefore this paper as well) however it is unclear whether it is possible to generalize any of the other convolutions in a similar way.

\subsection{Structure}\label{sec:structure}

The paper is structured as follows: 
in the remaining introductory sections, we attempt to provide some concrete motivation for studying the rectangular additive convolution. 
In Section~\ref{sec:convolution}, we define the rectangular additive convolution $\mysum{\dd}{\kk}$ and prove some of the basic properties that it satisfies.
Of particular importance in application is Theorem~\ref{thm:rr}, which shows that $\mysum{\dd}{\kk}$ preserves nonnegative rootedness.
In Section~\ref{sec:measure}, we introduce two equivalent methods --- the 
$H$-transform and $\Wtrans{\nn}{\uu}{}$ operator --- that we will use to 
measure the effect of 
$\mysum{\dd}{\kk}$ on the largest root (with respect to the input polynomials).

The proof of Theorem~\ref{thm:main} will be presented in Section~\ref{sec:proof}, however it will require a number of lemmas and simplifications that will need to be proved along the way.
Section~\ref{sec:inductions} contains two ``induction'' lemmas that will allow us to reduce the proof of Theorem~\ref{thm:main} to a subset of polynomials we call basic polynomials (see Section~\ref{sec:convolution} for definitions).
One of the main ingredients in the inductions is a ``pinching'' technique similar to the one used in \cite{convolutions4} which is introduced in Section~\ref{sec:pinch}.

Given these induction lemmas, the primary difficulty remaining will be in proving various ``base cases'' of Theorem~\ref{thm:main}.
These will be addressed in Section~\ref{sec:proof} modulo one major assumption: a bound on the Cauchy transform of a certain class of orthogonal polynomials (the {\em Gegenbauer polynomials}).
The missing bound will then be proved in Section~\ref{sec:geg}, along with a number of other nonasymptotic inequalities concerning Gegenbauer polynomials that may be of independent interest.
Section~\ref{sec:geg} has been (to a large extent) quarantined from the rest of the paper in an attempt to allow readers with a primary interest in orthogonal polynomials to find it accessible without needing to read other parts of the paper.

In Section~\ref{sec:examples}, we attempt to give some added 
understanding regarding the action of the (somewhat opaque) operator 
$\Wtrans{\nn}{\uu}{}$ through a set of informative examples.
Finally, in Section~\ref{sec:conclusion}, we end with some open problems.

\subsection{Motivation for this work}\label{sec:BG}

Before beginning, we would like to give some motivation for the study of this 
particular convolution.  
In particular, Section~\ref{sec:free_prob} and Section~\ref{sec:ffp} will 
discuss the relationship to a field first initiated by Voiculescu known as 
``free probability'' \cite{Voiculescu}.
The discussion of free probability will be restricted to these sections, and 
the remainder of the paper should be accessible to those unfamiliar with this 
area.
However, for those interested in (or at least aware of) the connections of 
polynomial convolutions to linear algebra (in particular, expected 
characteristic polynomials) may find these sections useful for understanding 
the motivations for some of what will appear.
For those interesting in learning more about this subject (and, in particular, 
the large role that combinatorics plays), we suggest 
starting with \cite{freebook}.

\subsubsection{As an operation on singular values}\label{sec:singular_values}

We start by recalling the relationship between the symmetric additive 
convolution from \cite{convolutions4} and eigenvalues of matrices.
There is a natural correspondence between any degree $\dd$ real-rooted 
polynomial $p$ and the class of $\dd \times \dd$ Hermitian matrices $H$ for  
which 
\[
p(x) = \mydet{x \ident - H}.
\]
It was shown in \cite{convolutions4} that the symmetric additive convolution of 
two degree $\dd$ real-rooted polynomials is (again) a degree $\dd$ real-rooted 
polynomial.
Hence the symmetric additive convolution can be viewed as a binary operation on 
these classes of matrices. 
This may seem coincidental, but it was shown in \cite{ffmain} that the 
symmetric additive convolution can be reproduced by the actual addition of 
matrices; that is, for all Hermitian $H$ and $K$, there exists a unitary matrix 
$U$ for 
which 
\[
\mydet{x \ident - H} \mysum{\dd}{} \mydet{x \ident - K} = \mydet{x \ident - (H 
+ U^* K U)}.
\]
Hence, for example, the symmetric additive convolution must satisfy Horn's 
inequalities \cite{tao} (precisely where in the Horn polytope the convolution 
lies is an interesting open question).

In a similar spirit, a polynomial in $\rrpos{\dd}$ can be written as $\mydet{x 
\ident - AA^*}$ for some $\dd \times (\dd + \nn)$ matrix $A$ (where we are free 
to choose $\nn \geq 0$).
In Theorem~\ref{thm:rr}, we show that the rectangular additive convolution of 
two polynomials in $\rrpos{\dd}$ is another polynomial in $\rrpos{\dd}$.
Hence, similar to above, the rectangular additive convolution can be seen as a 
binary operation on the classes of matrices with a given set of singular values.
The one small issue in that the roots of the polynomial $\mydet{x \ident - 
AA^*}$ are actually the {\em squares} of the singular values of $A$, and one 
can view the appearance of the $\Strans$ operator in the definition of 
$\Wtrans{\nn}{\uu}{}$ as a correction to this issue (see 
Remark~\ref{rem:WvsU}).

While this relationship to matrices is a primary motivation for defining (and 
studying) this particular convolution, we will see that (similar to 
\cite{convolutions4}) that it is beneficial to ignore this relationship when 
proving bounds.
In particular, one part of the induction we will use requires us to be able to 
apply the rectangular additive convolution to polynomials of different degrees, 
which would correspond to adding matrices of different sizes in the view 
discussed above. 
Hence we will refrain from using this view for the majority of this paper, 
instead working strictly in the realm of polynomials.
The two exceptions to this are Remark~\ref{rem:WvsU} and 
Section~\ref{sec:examples}, where we will use the connection to matrices to 
help explain certain aspects of our analysis that may not obvious when viewing 
the rectangular 
additive convolution completely in terms of polynomials.  
In particular, none of the main results will require any previous knowledge of 
free probability (or finite free probability).

\subsubsection{Inspiration from free probability}\label{sec:free_prob}

Much of the connection between polynomial convolutions and free probability 
stems from Remark~\ref{rem:WvsU}; in many ways, both theories consist of the 
only reasonable way to define a unitarily invariant binary operations on 
matrices.  
Hence the fact that polynomial convolutions turn into free convolutions (in the appropriate limit) is not surprising (proving that this is the case is less straight-forward \cite{ffmain}).
While the concepts in \cite{convolutions4, IF3, IF4} were developed without any knowledge of this connection, more recent work (including this paper) has benefited greatly from this relationship.

The connection is perhaps best seen by recalling that the primary tool used in \cite{convolutions4} to understand the behavior of the symmetric additive convolution was the differential transform
\[
\Utrans{p} = p - \uu p'
\]
for $\uu \geq 0$. 
This is nothing more than a more ``polynomial friendly'' version of the Cauchy transform (of the uniform distribution on the roots of $p$):
\[
\cauchy{p}{x} 
= \frac{1}{\deg(p)} \frac{p'(x)}{p(x)}
= \frac{1}{\deg(p)} \sum_{i} \frac{1}{x - \lambda_i(p)}
\]
where $\lambda_i(p)$ denote the roots of $p$.
In particular, one can check that
\[
\maxroot{\Utrans{p}} = t 
\iff
\cauchy{p}{t} = \frac{1}{\uu \deg(p)}
\]
and so many of the properties of $\Utrans{}$ can be derived directly from well known properties of the Cauchy transform.
One the other hand, $\Utrans{}$ has the advantage of remaining in the realm of real rooted polynomials (it is well known that the operator $(1 - \partial)$ preserves real rootedness, see for example \cite{BB2}) and this turns out to be useful in the analysis done in \cite{convolutions4}.

In order to understand the behavior of the rectangular convolution, we will do something similar.
Instead of relating to the Cauchy transform, however, we will use a construction inspired by the work of Benaych-Georges \cite{BG}.
The $H$-transform defined in this paper is a slight modification of what Benaych-Georges calls a {\em rectangular Cauchy transform} with much of the modification coming from the fact that the objects of interest in \cite{BG} are infinite dimensional operators, and so one must view the relationship between $\nn$ and $\dd$ as a ratio.
Some {\em a posteriori} explanations of the differences between $U_\alpha$ and $\Wtrans{\nn}{\uu}{}$ are discussed in Remark~\ref{rem:WvsU}, but this is not intended to obscure the fact that all of our {\em a priori} inspiration came directly from \cite{BG}.

\subsubsection{Finite free probability}\label{sec:ffp}

The original inspiration for \cite{convolutions4, ffmain} came from the study 
of expected characteristic polynomials of matrices.
In fact the definition of the rectangular additive convolution in 
Section~\ref{sec:convolution} was originally derived from the following 
observation \cite{dui}:

\begin{theorem}
Let $A$ and $B$ be $m \times n$ rectangular matrices, with (singular value) 
characteristic polynomials
\[
p(x) = \mydet{ x I - AA^*}
\AND
q(x) = \mydet{ x I - BB^*}
\]
Let $dQ$ and $dR$ denote Haar invariant measures over $U(m)$ and $U(n)$ (the 
spaces of unitary matrices of size $m$ and $n$, respectively).  
Then the coefficients of the polynomial 
\[
r(x) = \int_{Q, R} \mydet{ x \ident - (A + Q B R)(A + Q B R)^*} dQ dR
\]

are multilinear functions of the coefficients of $p$ and $q$.
\end{theorem}

\noindent One can check that the exact formula for $r(x)$ in terms of $p(x)$ 
and $q(x),$ derived in \cite{dui}, is given by the rectangular additive 
convolution.
In this respect, this paper can be viewed as the continuation of an 
investigation into the correspondence between expected characteristic 
polynomials and free probability that has become known as ``finite free 
probability."  

Finite free probability has played an important role in some recent results in 
computer 
science \cite{IF4, xie_xu}, and while we will not discuss any applications in 
this paper, we should mention that this paper is the first in a series of three 
papers.
The second will be an extension of \cite{IF4} (which showed the existence of bipartite Ramanujan graphs for any number of vertices $n$ and degree $d$) to the case of biregular, bipartite graphs \cite{next}.
The third will be an analogue of the analysis done in \cite{ffmain}, showing that the inequality in Theorem~\ref{thm:main} becomes an equality in the appropriate limit by showing that the individual terms converge to Benaych-Georges' rectangular $R$-transform \cite{nextnext}.
The third paper, in particular, suggests that the rectangular additive 
convolution can be viewed as a finite version of the addition of freely 
independent rectangular matrices.
In this respect, the bounds given in this paper show that the roots of $r(x)$ lie inside the support of the free convolution, a property which we will combine with the ``method of interlacing polynomials'' introduced in \cite{IF1} to build Ramanujan graphs.
%
%

\section{The Convolution}\label{sec:convolution}

\newcommand{\pol}{\operatorname{pol}}

For $\jj \geq 1$, we define $\rrpos{\jj}$ to be the collection of degree $\jj$ polynomials with real coefficients that have
\begin{enumerate}
\item All nonnegative roots
\item At least one root positive
\item The coefficient of $x^j$ positive
\end{enumerate}
and set $\rrpos{\jj}$ to be the (constant) $0$ polynomial for all $\jj \leq 0$.
Note that the second property only serves to eliminate the polynomial $x^\jj$ and that this is the only polynomial that is added by the closure:
\[
\overline{\rrpos{\jj}} = \rrpos{\jj} \cup  \{ x^\jj \}.
\]
We will use $\rrposle{\jj}$ to denote the union $\bigcup_{k=1}^\jj \rrpos{k}$.
Note that this {\em does not} include the $0$ polynomial.

We will call a polynomial $p \in \rrpos{\jj}$ {\em basic} if it has the form $p(x) = c (x - t)^\jj$ for some real numbers $c, t > 0$.
Otherwise we call it {\em nonbasic}.
Note that, trivially, every polynomial in $\rrpos{1}$ is basic.

We define a binary operation on polynomials $\mysum{\dd}{\kk}$ as follows:
\begin{definition}
For $i, j \leq \dd$ and $\kk \geq 0$, we define the {\em rectangular additive convolution} to be the linear extension of the operation
\begin{equation}\label{eq:defrac}
[x^i \mysum{\dd}{\kk} x^j](y) = 
\begin{cases}
\frac{(\kk + i)!(\kk + j)! j! i!}{(\dd+\kk)!\dd!(i + j - \dd)!(\kk + i + j - \dd)!} y^{i + j - \dd} & \text{when $i + j \geq \dd$} \\
0 & \text{otherwise}
\end{cases}
\end{equation}
In particular, if we write
\[
p(x) 
= \sum_{i} x^{\dd-i} a_i
\AND
q(x) 
= \sum_{i} x^{\dd-i} b_i
\]
then we have
\begin{equation}\label{eq:alt_def}
[p \mysum{\dd}{\kk} q](x)
= \sum_{\ell} x^{\dd-\ell} \left( \sum_{i+j=\ell}
\frac{(\dd-i)!(\dd-j)!}{\dd!(\dd-\ell)!}
\frac{(\kk+\dd-i)!(\kk+\dd-j)!}{(\kk+\dd)!(\kk+\dd-\ell)!} 
a_i b_j \right).
\end{equation}

\end{definition}

There are two special cases worth mentioning: for any polynomial $p \in \rrposle{d}$, we have
\begin{enumerate}
\item $[p \mysum{\dd}{\kk} x^\dd] = p$, and
\item $[p \mysum{\dd}{\kk}  x^{\dd-1} ] = \frac{1}{(\dd+\kk)\dd}( x p''(x) + (\kk + 1)p'(x))$.
\end{enumerate}

One property that can be derived directly from the definition is the following:
\begin{lemma}\label{lem:reduce}
Let $p \in \rrposle{\dd}$ and $q \in \rrposle{\dd-1}$.
Then 
\[
[ p \mysum{\dd}{\kk}  q ] = \frac{1}{\dd (\dd+\kk)} [ [p \mysum{\dd}{\kk}  x^{\dd-1} ] \mysum{\dd-1}{\kk}  q ]
\]
\end{lemma}
\begin{proof}
By linearity, it suffices to prove this when $p$ is a monomial of degree at 
most $\dd$ and $q$ a monomial of degree at most $\dd-1$.
Using (\ref{eq:defrac}) on $p(x) = x^j$, we have
\[
[ x^j \mysum{\dd}{\kk} x^{d-1} ] = 
\begin{cases}
\frac{j(\kk+j)}{\dd(\kk+\dd)} x^{j-1} & \text{for $0 < j \leq d$} 
\\
0 & \text{for $j = 0$}
\end{cases}
\]
and then the lemma follows by plugging into (\ref{eq:defrac}).
\end{proof}

\subsection{Preservation of nonnegative real roots}

The main property that we will need in this paper is the fact that the rectangular additive convolution (with the appropriate parameters) preserves the property of having all nonnegative roots.

\begin{theorem}\label{thm:rr}
Let $1 \leq i \leq j \leq \dd$.
Then for any $p \in \overline{\rrpos{i}}$ and $q \in \overline{\rrpos{j}}$, we have $[p \mysum{\dd}{\kk} q](x) \in \overline{\rrpos{i + j - \dd}}$.
\end{theorem}

The proof of this theorem will rely on a result of Lieb and Sokal \cite{LiebSokal}.
Recall that a polynomial $f$ is called {\em stable} if $f(z_1, 
\dots, z_k) \neq 0$ whenever $\Im(z_i) > 0$ for all $i$.
If (in addition) the coefficients of $f$ are real numbers, then $f$ is called {\em real stable}.
A degree $d$ univariate polynomial is stable if and only if it has $d$ real 
roots.
A multivariate polynomial is called {\em multiaffine} if it has degree at most $1$ in each of its variables.

\begin{theorem}[Lieb--Sokal]\label{thm:liebsokal}
If $f(y_1, \dots, y_n)$ and $g(x_1, \dots, x_n)$ are multiaffine real stable polynomials, then 
\[
f(-\partial_1, \dots, -\partial_n)\{  g(x_1, \dots, x_n) \} 
\]
is either the $0$ polynomial or is multiaffine real stable.
\end{theorem} 

We will also need the following well-known lemma:
\begin{lemma}\label{lem:xy}
For any degree $\jj$ univariate polynomial $p$ with positive leading coefficient, we have $p \in \overline{\rrpos{\jj}}$ if and only if $p(xy)$ is real stable.
\end{lemma}
\begin{proof}
Assume $p(x y)$ is not real stable; that is, $p(x y) = 0$ for some $x, y$ with $\Im(x) > 0$ and $\Im(y) > 0$.
Without loss of generality, we can set $x = r e^{\ii \theta_1}$ and $y = R e^{\ii \theta_2}$ where $0 < \theta_1 \leq \theta_2 < \pi$.
Hence $xy = rR e^{\ii\theta_3}$ is a root of $p$, and since
\[
\theta_3 = \theta_1 + \theta_2 \in (0, 2 \pi)
\]
this root is not a nonnegative real, so $p \notin \overline{\rrpos{\jj}}$.

In the other direction, if $p \notin \overline{\rrpos{\jj}}$ then $p(x^2)$ has 
a non-real root.
Since $p$ has real coefficients, the roots come in conjugate pairs so one such root must be in the upper half plane.  
Setting $x$ and $y$ to be that root shows that $p(x y)$ is not real stable.
\end{proof}

We will write $\pol_x^{k}$ to denote the map 
\[
\pol_x^k[ x^{j} ] = \frac{\sigma_j(x_1, \dots, x_k)}{\binom{k}{j}}
\]
where $\sigma_j$ is the $j^{th}$ elementary symmetric polynomial on the inputs 
$x_1, \dots, x_k$ (and is undefined when $j > k$).
The linear extension of the map $\pol_x^i$ to polynomials is known as the {\em polarization} operator.
It should be clear that $\pol_x^i[p(x)]$ is multiaffine for any polynomial $p$ and any integer $i$.
A theorem of Borcea and Br{\"{a}}nd{\"{e}}n shows that polarization preserves the property of being real stable \cite{BB2}:

\begin{theorem}\label{thm:polarize}
If $p(x, y_1, \dots, y_n)$ is a real stable polynomial then $\pol_x^i[p]$ is real stable as well.
\end{theorem}

\begin{proof}[Proof of Theorem~\ref{thm:rr}]
We first note that it suffices to consider the case when $i = j = \dd$.
To see why, we can proceed by induction on $\dd$.
Since $i = j = 1$ is the only possibility when $\dd = 1$, the base case will be covered.
Now for $\dd > 1$, if $i < \dd$ or $j < \dd$, then we can use Lemma~\ref{lem:reduce} to write the same polynomial using a convolution with parameter $\dd-1$, and this is real rooted by the induction hypothesis.

To ease notation, we will write
\[
\theta^i_x := \sigma_i(x_1, \dots, x_{\dd})
\AND
\theta^i_y := \sigma_i(y_1, \dots, y_{\kk+\dd}).
\]
and
\[
\delta^i_x := \sigma_i(\partial_{x_1}, \dots, \partial_{x_{\dd}})
\AND
\delta^i_y := \sigma_i(\partial_{y_1}, \dots, \partial_{y_{\kk+\dd}}).
\]
Given polynomials $p, q \in \overline{\rrpos{\dd}}$, with $p(x) = \sum_i x^{\dd-i}a_i$ and $q(x) = \sum_i x^{\dd-i}b_i$, set
\[
f(x_1, \dots, x_{\dd}, y_1, \dots, y_{\dd+\kk}) 
= \pol_x^{\dd} \left[ \pol_y^{\dd+\kk} \left[ y^{\kk} p(x y)\right] \right]
= \sum_i \frac{\theta^{\dd-i}_x}{\binom{\dd}{i}} \frac{\theta^{\kk+\dd-i}_y}{\binom{\kk+\dd}{i}} a_i
\]
and 
\[
g(x_1, \dots, x_{\dd}, y_1, \dots, y_{\dd+\kk}) 
= \pol_x^{\dd} \left[ \pol_y^{\dd+\kk} \left[ (x y)^{\dd} q\left(\frac{1}{x y}\right) \right] \right]
= \sum_i \frac{\theta^{i}_x}{\binom{\dd}{i}} \frac{\theta^{i}_y}{\binom{\kk+\dd}{i}} b_i
\]
Both $f$ and $g$ are multiaffine (by definition of polarization) and real stable (by a combination of Lemma~\ref{lem:xy} and Theorem~\ref{thm:polarize}).
Hence Theorem~\ref{thm:liebsokal} implies that
\begin{align*}
h(x_1,\dots, x_{\dd}, y_1, \dots, y_{\dd+\kk}) 
&= g(-\partial_{x_1}, \dots, -\partial_{x_{\dd}}, -\partial_{y_1}, \dots, -\partial_{y_{\kk+\dd}}) \{  f(x_1, \dots, x_{\dd}, y_1, \dots, y_{\dd+\kk}) \} 
\end{align*}
is either $0$ or real stable.
If it is $0$, then we are done, so assume not.
Since substitution of variables preserves real stability, the bivariate polynomial $r(x, y)$ formed by substituting $x \leftarrow x_i$ and $y \leftarrow y_i$ into $h$ is also real stable.
We claim that 
\[
r(x, y) = y^{\kk} [p \mysum{\dd}{\kk} q](x y)
\]
which, by Lemma~\ref{lem:xy}, would complete the proof.

The main observation that we will need to compute $h$ is that for any $\ell > 0$, 
\[
  \sigma_i(\partial_{z_1}, \dots, \partial_{z_\ell}) \{ \sigma_j(z_1, \dots, z_\ell) \}
 = \begin{cases}
 \binom{\ell+i-j}{i} \sigma_{j-i}(z_1, \dots, z_\ell) & \textrm{for $i \leq j$}\\
  0 & \textrm{otherwise}
\end{cases}
\]
which can easily be checked by hand.  
In particular, this gives 
\[
\delta^i_x \theta^{\dd-j}_x = \binom{i+j}{i} \theta^{\dd-i-j}_x
\AND
\delta^{i}_y \theta^{\kk+\dd-j}_y 
= \binom{i+j}{i} \theta^{\kk+\dd-i-j}_y
\]
when $i \leq j$ (and $0$ otherwise).
Applying this to our formula for $h$, we get
\begin{align*}
h(x_1,...x_{\dd},y_1,...,y_{\dd+\kk})
&= \sum_{i=0}^{\dd} b_i (-1)^{2i}\frac{\delta^i_x}{\binom{\dd}{i} } \frac{ \delta^{i}_y}{\binom{\dd+\kk}{i}} \sum_{j=0}^{\dd} a_j \frac{\theta_x^{\dd-j}}{\binom{\dd}{j} } \frac{ \theta_y^{\kk+\dd-j}}{\binom{\kk+\dd}{j}} 
\\&= \sum_{i=0}^{\dd}  \sum_{j=0}^{\dd} a_j b_i \frac{\binom{i+j}{i}\theta_x^{\dd-i-j}}{\binom{\dd}{i}\binom{\dd}{j}} 
\frac{\binom{i+j}{i}\theta_y^{\kk+\dd - i - j}}{\binom{\kk+\dd}{j}\binom{\kk+\dd}{i}} 
\end{align*}
Plugging in $x \leftarrow x_i$ and $y \leftarrow y_i$, we then get
\begin{align*}
r(x,y)
&= \sum_{\ell=0}^{\dd} \sum_{i+j=\ell}  a_j b_i 
\frac{\binom{\ell}{i}\binom{\dd}{\ell} x^{d-\ell}}{\binom{\dd}{i}\binom{\dd}{j}} 
\frac{\binom{\ell}{i}\binom{\dd+\kk}{\ell} y^{\dd+\kk - \ell}}{\binom{\dd+\kk}{j}\binom{\dd+\kk}{i}} 
\\&= y^\kk \sum_{\ell=0}^{\dd}  (x y)^{\dd-\ell} \left( \sum_{i+j=\ell} a_j b_i 
\frac{(\dd-i)!(\dd-j)!}{\dd!(\dd-\ell)!}\frac{(\kk+\dd-i)!(\kk+\dd-j)!}{(\kk+\dd)!(\kk+\dd-\ell)!}\right)
\\&= y^{\kk} [p \mysum{\dd}{\kk} q](x y).
\end{align*}
\end{proof}

\section{The measuring stick}\label{sec:measure}

In this section, we define the $\Wtrans{\nn}{\uu}{}$ operator that will be used to measure the effects of the rectangular additive convolution on the largest root of a polynomial.
Rather than define it directly, however, it will be useful to first introduce a modification of the $H$-transform from \cite{BG}, which will then have $\Wtrans{\nn}{\uu}{}$ as its corresponding differential operator.
To state the two succinctly, it will help to introduce two other operators:
\begin{equation}\label{eq:SandV}
[\Strans{p}](x) = p(x^2)
\AND
[\Vtrans{\nn}{p}](x) = x^{\nn} p(x).
\end{equation}

Given a polynomial with nonnegative roots, we define the {\em $H$-transform of $p$ (with parameter $\nn$)} as
\begin{equation}\label{def:htrans}
\Htrans{\nn}{p}(x)
= [\log \Strans{p}]' [\log \Strans{\Vtrans{\nn}{p}}]'
= \frac{[\Strans{p}]'}{\Strans{p}} \left( \frac{2\nn}{x} + \frac{[\Strans{p}]'}{\Strans{p}} \right)
\end{equation}
and the corresponding differential operator
\begin{equation}\label{def:wtrans}
\Wtrans{\nn}{\uu}{p}
= [\Strans{p}][\Strans{\Vtrans{\nn}{p}}] - \uu^2  [\Strans{p}]'[\Strans{\Vtrans{\nn}{p}}]'
.
\end{equation}
Note that the parameter $\nn$ is intended to be the same as the parameter $\kk$ in (\ref{eq:defrac}).
One could consider $\Htrans{k}{p}$ and $\Wtrans{k}{\uu}{p}$ for general value of $k$, but we will only be using these transforms to directly measure the effect of $\mysum{\dd}{\kk}$ and so there is no reason to consider this more general case.

\begin{remark}\label{rem:WvsU}
For those familiar with the $\Utrans{}$ operator from \cite{convolutions4}, we 
would like to point out that there is a natural way to understand the 
differences between that operator and the $\Wtrans{\nn}{\uu}{}$ defined here  
using the relationship to characteristic polynomials 
mentioned in Section~\ref{sec:singular_values}.
As discussed there, the symmetric additive convolution can be viewed as a
binary operation on eigenvalues (technically, classes of Hermitian matrices 
with the same eigenvalues) and, in a similar spirit, the rectangular additive 
convolution can be viewed as a binary operation on singular values 
(technically, classes of matrices with the same dimensions and singular values).
This correspondence goes via the characteristic polynomials
$
\mydet{x \ident - H}
$
with $H$ Hermitian (in the eigenvalue case) and 
$
\mydet{x \ident - AA^*}
$
with $A$ a $\dd \times (\nn + \dd)$ matrix with $n \geq 0$ (in the singular 
value case).
The second correspondence is not perfect, however, in that the roots of the 
polynomial $\mydet{x \ident - AA^*}$ are actually the {\em squares} of the 
singular values of $A$.

The first main difference between the $\Utrans{}$ and $\Wtrans{\nn}{\uu}{}$ 
operators --- the appearance of the $\Strans$ operator --- can be viewed 
as a necessary compensation for this issue.
One concern one might have is that this ``correction'' effectively creates two copies of each singular value (a positive one and a negative one).
Fortunately, our primary interest is in understanding the largest singular value (that is, the largest root of $\Strans p$), and so the addition of extra negative roots will prove to be irrelevant in our analysis.

The remaining differences between the $\Utrans{}$ and $\Wtrans{\nn}{\uu}{}$ 
operators --- the quadratic nature and appearance of the $\Vtrans{\nn}{}$ 
operator (and the parameter $\nn$ in general) --- can likewise be seen as a 
necessary compensation for a different issue.
Given the polynomial $\mydet{x \ident - AA^*}$, there is no way to know the 
number of columns in the underlying matrix $A$.  
One might hope that this is irrelevant, but the analysis in \cite{BG} shows 
that this is not the case.
One obvious attempt to correct this would be to instead consider the 
polynomial $\mydet{x \ident - A^* A}$, but one runs into the same problem 
concerning the number of rows (and therefore the value of $\dd$).
The $\Wtrans{\nn}{\uu}{}$ operator compensates for this by using {\em both} 
polynomials
\[
p(x) = \mydet{x\ident - A A^*}
\AND
\Vtrans{\nn}{p}(x) = \mydet{x\ident - A^* A}
\]
to be used.
The one case where the $\Wtrans{\nn}{\uu}{}$ operator would not need 
both polynomials is of course when $\nn 
= 0$.
In this case, $\maxroot{\Wtrans{0}{\uu}{p}}$ becomes a difference of squares (which can then be factored):
\begin{align*}
\maxroot{\Wtrans{0}{\uu}{p}} 
= \maxroot{ p(x^2)^2 - 4 x^2 \uu^2 p'(x^2)^2 } 
&= \maxroot{ p(x^2) - 2 \uu x p'(x^2) } 
\\&= \maxroot{ U_{\alpha} \Strans{p} }.
\end{align*}
This is the reason that the analysis of the asymmetric convolution in 
\cite{convolutions4} could be done using the simpler $U_\alpha$ operator 
(despite being a special case of the results in this paper).
\end{remark}

As in \cite{convolutions4}, we would like to be able to associate the function $\Htrans{\nn}{p}$ with the largest root of the polynomial $\Wtrans{\nn}{\uu}{p}$ but for this to be well-defined, we first need to show that $\Wtrans{\nn}{\uu}{p}$ has real roots.
Our proof will use a classical result in real rooted polynomials \cite{BB2}:
\begin{theorem}[Hermite--Biehler]
Let $f$ and $g$ be polynomials with real coefficients.
Then the following are equivalent:
\begin{enumerate}
\item Every root $\lambda$ of $f + \ii g$ satisfies $\Im\{\lambda \} \leq 0$ (here $\Im$ denotes the imaginary part)
\item $f$ and $g$ are real rooted and the polynomial $f g' - g f' \leq 0$ at any point $x \in \R$.
\end{enumerate}
\end{theorem}

Using this, it is easy to show the following lemma, which will immediately imply what we need:
\begin{lemma}\label{lem:rr}
Let $p$ and $q$ be real rooted polynomials.
Then 
\[
p(x) q(x) - p'(x) q'(x)
\]
is real rooted.
\end{lemma}
\begin{proof}
It is an easy consequence of Rolle's theorem that $p'$ is real rooted whenever $p$ is, and then one can check that 
\[
p(x) p''(x) - p'(x)^2 \leq 0
\]
for all $x$ by noticing
\[
\frac{p(x) p''(x) - {p'(x)}^2}{p(x)^2} 
= ( \log p(x) )'' 
=  \sum_i -\frac{1}{(x-\lambda_i(p))^2}.
\]
Hermite--Biehler then implies that the polynomials 
\[
p(x) + \ii p'(x) 
\AND
q(x) + \ii q'(x) 
\]
have no roots in the upper half plane.
Hence their product
\[
pq + \ii (p'q + q'p) - p'q'
\]
has no roots in the upper half plane as well.
So Hermite--Biehler (in the opposite direction) gives that $pq - p'q'$ is real rooted.
\end{proof}

We therefore have the following correspondence, which amounts to nothing more than algebraic manipulation:
\begin{corollary}\label{lem:translate}
\[
\Htrans{\nn}{p}(t) = \frac{1}{\uu^2}
\iff
\maxroot{\Wtrans{\nn}{\uu}{p}} = t
\]
\end{corollary}

\subsection{Monotonicity Properties}\label{sec:monotonicity}

One of the most important properties of the function $\cauchy{p}{x}$ when $p$ is real rooted is that it is strictly monotone decreasing on the interval $[\maxroot{p}, \infty)$ (something that can be seen directly from the definition).
Our new functions will inherit that property:

\begin{lemma}\label{lem:Hmonotone}
For all polynomials $p$ with nonnegative roots, the function $\Htrans{\nn}{p}(x)$ is strictly decreasing at any $x > \sqrt{\maxroot{p}}$.
\end{lemma} 
\begin{proof}
Let $t = \maxroot{p}$.
Since $p$ has nonnegative roots, both $\Strans{p}$ and $\Strans{\Vtrans{\nn}{p}}$ are real rooted with 
\[
\maxroot{\Strans{p}} = \maxroot{\Strans{\Vtrans{\nn}{p}}} = \sqrt{t}.
\]
Hence both $\cauchy{\Strans{p}}{x}$ and $\cauchy{\Strans{\Vtrans{\nn}{q}}}{x}$ are strictly decreasing for $x \geq \sqrt{t}$ and so the product is strictly decreasing as well.
\end{proof}

This then implies two ``inverse'' properties for the differential operators:

\begin{corollary}  \label{lem:monotU}
Let $p$ be a polynomial with nonnegative roots.
Then for any $\uu > 0$, the quantity
\[
r(\uu) := \maxroot{\Wtrans{\nn}{\uu}{p}}
\]
is strictly increasing.
\end{corollary}
\begin{proof}
By Corollary~\ref{lem:translate}, we have
\[
\Htrans{\nn}{p}(r(\uu)) = \frac{1}{\uu^2}.
\]
Taking derivatives (in $\uu$) on both sides gives
\[
[\Htrans{\nn}{p}]'(r(\uu)) r'(\uu) = \frac{-1}{\uu^3} < 0
\]
where we know from Lemma~\ref{lem:Hmonotone}  that $[\Htrans{\nn}{p}]'(r(\uu)) < 0$.
Hence it must be that $r'(\uu) > 0$.
\end{proof}

The one disadvantage of these new differential operators is their quadratic nature.
In the case where we will only need qualitative information about two polynomials, we will be able to reduce to the simpler case:

 \begin{lemma}\label{lem:simplify}
Let $p, q$ be polynomials with nonnegative roots.
Then for any $x > \sqrt{\maxroot{pq}}$, we have
\[
\Htrans{\nn}{p}(x) \leq \Htrans{\nn}{q}(x) 
\iff 
[ \log p ]'(x^2) \leq [ \log q ]'(x^2).
\]
\end{lemma}
\begin{proof}
Let 
\begin{equation}\label{wp}
w_p = [\log \Strans{p}]'(x) = 2 x \frac{p'(x^2)}{p(x^2)} 
\end{equation}
and similarly for $w_q$.
By definition, we therefore have
\[
\Htrans{\nn}{p}(x) = \frac{n}{x} w_p + w_p^2
\AND
\Htrans{\nn}{q}(x) = \frac{n}{x} w_q + w_q^2
\]
and so the difference is
\[
\Htrans{\nn}{p}(x) - \Htrans{\nn}{q}(x) = 
(w_p - w_q) \left( \frac{n}{x} + w_p + w_q \right)
\]
where for $x > \maxroot{p q}$ we have $\frac{n}{x}, w_p, w_q > 0$.
Hence 
\[
\Htrans{\nn}{p}(x) - \Htrans{\nn}{q}(x) 
\AND
w_p - w_q
\]
have the same sign.
But by (\ref{wp}), we can write
\[
 w_p - w_q = 2 x \frac{p'(x^2)}{p(x^2)} - 2 x \frac{q'(x^2)}{q(x^2)}
\]
and since $x > 0$, the lemma follows.
\end{proof}

Using Lemma~\ref{lem:simplify}, the primary inequality that we will need becomes a simple calculation.

\begin{lemma}\label{lem:quasilinear}
Let $p, q, r$ be polynomials with nonnegative roots and positive leading coefficients such that $p + q = r$.
Furthermore, assume each has at least one positive root.
Then for all $n \geq 0$ and all $\uu > 0$, we have
\[
\min \left\{ \Htrans{\nn}{p}(t), \Htrans{\nn}{q}(t) \right \}
\leq
\Htrans{\nn}{r}(t)
\leq 
\max \left\{ \Htrans{\nn}{p}(t), \Htrans{\nn}{q}(t) \right \}
\]
whenever $t > \sqrt{\maxroot{p q r}}$.
Furthermore, equality holds in one direction if and only if it holds in the other direction.
\end{lemma}
\begin{proof}
By Lemma~\ref{lem:simplify}, it suffices to show that 
\[
\min \left\{ [ \log p ]'(s) , [ \log q ]'(s) \right \}
\leq
[ \log r ]'(s)
\leq
\max \left\{ [ \log p ]'(s) , [ \log q ]'(s) \right \}
\]
for any $s$ which is greater than $\maxroot{p q r}$.
Now notice that
\[
\frac{r'}{r} 
= \frac{p' + q'}{r} 
= \frac{p}{r}\frac{p'}{p} + \frac{q}{r} \frac{q'}{q} 
\]
where $p / r$ and $q / r$ are both positive at $s$ (since $s$ is larger than the maximum root of each polynomial, and each polynomial has a positive leading coefficient).
So setting $\lambda = p(s)/r(s)$, we have
\[
[ \log r ]'(s) = \lambda [ \log p ]'(s) + (1- \lambda)[ \log q ]'(s)
\]
for some $0 < \lambda < 1$.
The lemma then follows.
\end{proof}
 
Of particular importance will be the case when $t = \maxroot{\Wtrans{\nn}{\uu}{p}}$.

\begin{corollary}\label{cor:straddle_roots}
Let $p, q, r$ satisfy the conditions of Lemma~\ref{lem:quasilinear} with  
\[
\beta_p = \maxroot{\Wtrans{\nn}{\uu}{p}}
\AND
\beta_q = \maxroot{\Wtrans{\nn}{\uu}{q}}
\AND
\beta_r = \maxroot{\Wtrans{\nn}{\uu}{r}}.
\]
Assume further that $\beta_r \geq \sqrt{\maxroot{p q r}}$.
Then 
\[
\min \left\{ \beta_p, \beta_q \right \}
\leq
\beta_r
\leq 
\max \left\{ \beta_p, \beta_q \right \}
\]
with two of the $\beta_k$ equal if and only if all three of the $\beta_k$ are equal.
\end{corollary}
\begin{proof}
Set $t = \beta_r$ in Lemma~\ref{lem:quasilinear} and assume (without loss of generality) that
\begin{equation}\label{eq:compare}
\Htrans{\nn}{p}(\beta_r) \leq \Htrans{\nn}{r}(\beta_r) \leq \Htrans{\nn}{q}(\beta_r)
\end{equation}
By Lemma~\ref{lem:translate}, we have 
\[
\Htrans{\nn}{p}(\beta_p) = \Htrans{\nn}{q}(\beta_q) = \Htrans{\nn}{q}(\beta_r) = \frac{1}{\uu^2}
\]
Hence we can rewrite (\ref{eq:compare}) as 
\[
\Htrans{\nn}{p}(\beta_r) \leq \Htrans{\nn}{p}(\beta_p)
\AND
\Htrans{\nn}{q}(\beta_q) \leq \Htrans{\nn}{q}(\beta_r)
\]
which, by Lemma~\ref{lem:Hmonotone} implies that 
\[
\beta_q \leq \beta_r \leq \beta_p 
\]
which satisfies the desired conclusion.
\end{proof}

\section{Inductions}\label{sec:inductions}

The goal of this section is to prove two ``induction steps'' that will be useful in the proof of Theorem~\ref{thm:main}.
Let
\begin{equation}\label{def:myphi}
\myphi{\nn}{\kk}{\dd}{\uu}{p}{q} := \Theta^{n}_{\uu}(p) + \Theta^{n}_{\uu}(q) - \Theta^{n}_{\uu}(p \mysum{\dd}{\kk} q) - (\kk + 2\dd)\uu
\end{equation}
where 
\[
\Theta^{\nn}_\uu(p) := \sqrt{\nn^2 \uu^2 + [\maxroot{\Wtrans{\nn}{\uu}{p}}]^2}.
\]
Notice that Theorem~\ref{thm:main} can be restated as saying that $\myphi{\nn}{\kk}{\dd}{\uu}{p}{q} > 0$.
Our first induction will be on the degree $\dd$.

\begin{lemma}\label{lem:inductiond}
Let $2 \leq \jj < \dd$.
Assume that, for some fixed $q \in \rrposle{\dd-1}$ and $\alpha > 0$, we have
\begin{enumerate}
\item[(i)] $\myphi{\nn}{\kk}{\dd-1}{\uu}{p}{q} > 0$ for 
all $p \in \rrpos{\jj-1}$
\item[(ii)] $\myphi{\nn}{\kk}{\dd}{\uu}{p}{x^{\dd-1}} > 0$ for 
all $p \in \rrpos{\jj}$
\end{enumerate}
Then $\myphi{\nn}{\kk}{\dd}{\uu}{p}{q} > 0$ for 
all $p \in \rrpos{j}$.
\end{lemma}

Our second induction will be on the number of distinct roots in the 
polynomial.


\begin{lemma}\label{lem:induction}
Let $2 \leq \jj \leq \dd$.
Assume that, for a fixed $q \in \rrposle{\dd}$ and $\uu > 0$, we have
\begin{enumerate}
\item[(i)] $\myphi{\nn}{\kk}{\dd}{\uu}{p}{q} > 0$ for 
all polynomials $p \in \rrpos{\jj-1}$
\item[(ii)] $\myphi{\nn}{\kk}{\dd}{\uu}{p}{q} > 0$ for 
all basic polynomials $p \in \rrpos{\jj}$
\end{enumerate}
Then $\myphi{\nn}{\kk}{\dd}{\uu}{p}{q} > 0$ for 
all polynomials $p \in \rrpos{\jj}$.
\end{lemma}

The proofs of both lemmas are given in Section~\ref{sec:proofinduct}.
The proof of Lemma~\ref{lem:induction} will utilize a decomposition of a nonbasic polynomial $p \in \rrpos{j}$ into two ``simpler'' polynomials that is proved in Section~\ref{sec:pinch}.
By ``simpler'', we mean that one of the polynomial will have lower degree 
(allowing for induction) and the other polynomial will have two of its roots 
moved closer together (an operation that is referred to in \cite{convolutions4} 
as ``pinching'').
Obviously one can always pinch a nonbasic polynomial --- the goal will be to find a pinch that decreases $\myphi{\nn}{\kk}{\dd}{\uu}{\cdot}{q}$, as this would imply that any minimum of $\myphi{\nn}{\kk}{\dd}{\uu}{\cdot}{q}$ (if it exists) must be a basic polynomial.

Because basic polynomials feature prominently in the computations moving forward, it will be useful to have calculated the following quantity beforehand:

\begin{lemma}\label{lem:basictheta}
Let $p(x) = c (x - \lambda)^j$ with $j \geq 1$ and $c, \lambda > 0$.
Then
\[
\Theta^n_\uu(p) = \sqrt{(n+j)^2 \uu^2 + \lambda} + \uu j.
\]
\end{lemma}
\begin{proof}
For $p$ as given, we have
\[
\Wtrans{n}{\uu}{p}(x)
=c (x^2-\lambda)^{2(j-1)}x^{2n} r_{\lambda}(x^2)
\]
where $r_{\lambda}(x)= (x - \lambda)^2 - 4 \jj(\nn+\jj) \uu^2 (x - \lambda) - 4 \jj^2\uu^2 \lambda$.
Hence
\begin{equation}\label{eq:basicmr}
\maxroot{\Wtrans{n}{\uu}{p}}^2 = \maxroot{r_\lambda} = \lambda + 2(n+j)j\uu^2 +2j \uu\sqrt{(n+j)^2 \uu^2+\lambda}
\end{equation}
and so 
\begin{align*}
\Theta^n_\uu(p)^2 
= n^2 \uu^2 +  \lambda + 2(n+j)j\uu^2 +2j \uu\sqrt{(n+j)^2 \uu^2+\lambda}
= \left( \sqrt{(n+j)^2 \uu^2 + \lambda} + \uu j \right)^2.
\end{align*}
\end{proof}

\subsection{The Pinch}\label{sec:pinch}

\newcommand{\Htransbar}[2]{\hat{\mathcal{H}}^{#1}_{#2}}
\newcommand{\Wtransbar}[3]{\hat{W}^{#1}_{#2}{#3}}

\newcommand{\uuu}{\zeta}

The main goal of this subsection is to prove Lemma~\ref{lem:recPinch}, which provides the existence of a pinch with the properties that will be useful in the proof of Lemma~\ref{lem:induction}.

\begin{lemma}\label{lem:recPinch}
Fix $\uu > 0$, $j \geq 2$, and let $p(x) \in \rrpos{j}$ be nonbasic.
Then there exist $\ptil \in \rrpos{j}$ and $\phat \in \rrpos{j-1}$
  so that 
\begin{enumerate}
\item[a.] $p (x) = \phat (x) + \ptil (x)$,
\item[b.] $\maxroot{\ptil} \leq \maxroot{p}$,
\item[c.] \label{eqn:recPinch} 
$
\maxroot{\Wtrans{n}{\uu}{\ptil} }
 = \maxroot{\Wtrans{n}{\uu} {\phat} }
 = \maxroot{\Wtrans{n}{\uu}{p} }
$
\item[d.] $\maxroot{\phat} \leq \maxroot{\Wtrans{n}{\uu}{p} }$.
\end{enumerate}
\end{lemma}

The proof of Lemma~\ref{lem:recPinch} will occur in two steps.
First we will prove a pinching lemma for a ``linearized'' version of the $W$-transform (by removing the $\Strans$ operator), and then we will show how the existence of a linearized pinch implies Lemma~\ref{lem:recPinch}.
Note that this is the only section where these linearized operators will appear.

We define the ``linearized'' versions of the $H$-transform as:
\[
\Htransbar{n}{p}:= [\log p]' \log[ \Vtrans{n}{p} ]' = [\log p]' \left( \frac{n}{x} + [\log p]' \right)
\]
and note that for two polynomials $p, q$, $[\log p]'(t) = [\log q]'(t)$ implies $\Htransbar{n}{p}(t) = \Htransbar{n}{q}(t)$.
Corresponding to this is the ``linearized'' version of the $W$-transform
\[
\Wtransbar{n}{\uuu}{p}(x) = [p][\Vtrans{n}{p}] - \uuu^2 [p]'[\Vtrans{n}{p}]'
\]
which by Lemma~\ref{lem:rr} is real-rooted as long as $p$ is real rooted.
Hence we can say 
\begin{equation}\label{lineartranslate}
\Htransbar{n}{p}(t) = \frac{1}{\uuu^2} \iff \maxroot{\Wtransbar{n}{\uuu}{p}} = t.
\end{equation}

\begin{lemma}[Linearized pinching]\label{lem:pinch}
Let $\uuu > 0$, $j \geq 2$, and let $p(x) = (x - a)(x - b) r(x) \in \rrpos{j}$ be a monic polynomial with $0 \leq a < b = \maxroot{p}$.
Then there exist real numbers $\mu \in (a, b)$ and $\rho > b$ so that the polynomials 
\[
  \ptil(x) = (x-\mu)^{2}r(x)
\AND
  \phat (x) = \kappa (x-\rho )r(x)
\]
satisfy
\begin{itemize}
\item [a.] $\ptil \in \rrpos{j}$ and $\phat \in \rrpos{j-1}$ (in particular, $\kappa > 0$)
\item [b.] $p = \ptil + \phat$
\item [c.] $\maxroot{\Wtransbar{n}{\uuu}{p}} = \maxroot{\Wtransbar{n}{\uuu}{\ptil}}=\maxroot{\Wtransbar{n}{\uuu}{\phat}}$
\item [d.] $\rho < \maxroot{\Wtransbar {n}{\uuu}{p}}$
\end{itemize}
\end{lemma}
\begin{proof}
Let $t = \maxroot{\Wtransbar{n}{\uuu}{p}}$.
Since $\uuu > 0$, we have by Lemma~\ref{lem:monotU} that $t > b > a$, so there is no issue in setting
\[
  \mu = t - \frac{2}{1/(t - a) + 1/(t - b)}.
\]
or (rearranging slightly)
\begin{equation}\label{harmonic}
  \frac{2}{t - \mu} = \frac{1}{t - a} + \frac{1}{t - b}.
\end{equation}
Note that (\ref{harmonic}) gives $(t - \mu)$ as the harmonic average of $(t - a)$  and $(t - b)$ (where $a \neq b$).
Using the fact that $t > b > a$, this implies that $(t - b) < (t - \mu) < (t - a)$ or (equivalently) $a < \mu < b$.
Furthermore, 
\[
[\log (x-c)(x-d)r(x) ]'(t) = 
\frac{r'(t)}{r(t)} + \frac{1}{t - c} + \frac{1}{t - d}
\]
so (\ref{harmonic}) also shows that $[\log p]'(t) = [\log \ptil]'(t)$.

Now define $\phat = p - \ptil$.
Since $p$ and $\ptil$ have the same leading coefficient, it should be clear that $\phat$ has degree $j-1$.
Furthermore, we can now write
\[
[\log p]'
= \frac{p'}{p} 
= \frac{\phat'}{p} + \frac{\ptil'}{p}
= \frac{\phat}{p} [\log \phat]' + \frac{\ptil}{p} [\log \ptil]'.
\]
Evaluating at the point $t$, we get that
\[
[\log p]'(t) = \lambda [\log \phat]'(t) + (1 - \lambda) [\log \ptil]'(t)
\]
for some $\lambda \in (0, 1)$, so since $[\log p]'(t) = [\log \ptil]'(t)$ we must have
\[
[\log p]'(t) 
= [\log \ptil]'(t) 
= [\log \phat]'(t).
\]
As is mentioned above, this implies 
\[
\Htransbar{n}{p}(t) = \Htransbar{n}{\ptil}(t) = \Htransbar{n}{\phat}(t)
\]
and so by (\ref{lineartranslate})
\[
\maxroot{\Wtransbar{n}{\uuu}{p}} 
= \maxroot{\Wtransbar{n}{\uuu}{\ptil}}
=\maxroot{\Wtransbar{n}{\uuu}{\phat}}.
\]

The fact that $a < \mu < b$ shows that $\ptil \in \rrpos{j}$, and so it remains to show that $\kappa > 0$ and that $\rho \in (b, t)$.
We first multiply out the equation $p = \ptil + \phat$ and equate coefficients to get
\begin{equation}\label{eq:coeff}
\kappa = 2\mu - a - b 
\AND
\rho = \frac{\mu^2 - a b}{2 \mu - a- b}.
\end{equation}
Now to see that $\kappa > 0$, recall that (\ref{harmonic}) expresses $t - \mu$ as the harmonic average of two postive numbers.
The inequality between the arithmetic and harmonic means therefore implies
\[
t - \mu < \frac{1}{2} (2 t - a - b)
\]
which, after rearranging, gives that $2 \mu > a + b$.

To see that $\rho \in (b, t)$, we can solve for $t$ in (\ref{harmonic}) to get
\[
t = \frac{\mu(a+b) - 2 a b}{2 \mu - a - b}
\]
and then use the formula for $\rho$ in (\ref{eq:coeff}) to compute
\[
\rho - b 
= \frac{(\mu-b)^2}{2 \mu - a- b} > 0
\AND
t - \rho = \frac{(b - t)(t - a)}{2 \mu - a - b} > 0
\]
as needed.
\end{proof}

To move from the ``linearized'' pinch to the ``quadratic'' pinch that we need, we will use the following observation:

\begin{lemma}\label{lem:quad}
Fix $\uu > 0$, $j \geq 2$, and let $p \in \rrpos{j}$.
Then 
\[
\maxroot{ \Wtransbar{n}{\uu}{p} } = t \iff \maxroot{\Wtrans{n}{2\uu t}{p} } = t^2
\]
\end{lemma}
\begin{proof}
By definition, we have
\begin{align*}
\maxroot{ \Wtransbar{n}{\uu}{p} } = t
&~\Longrightarrow~ \Wtrans{n}{\uu}{p}(t) = 0 
\\ &\iff [\Strans{q}](t)[\Strans{\Vtrans{n}{p}}](t) - \uu^2 [\Strans{p}]'(t)[\Strans{\Vtrans{n}{p}}]'(t) = 0
\\ &\iff p(t^2)[\Vtrans{n}{p}](t^2)- 4 t^2 \uu^2 [\Vtrans{n}{p})]'(t^2) p'(t^2) = 0
\end{align*}
and so $t^2$ is {\em some} root of $\maxroot{\Wtrans{n}{2\uu t}{p} }$; we wish to show it is the largest one.
Assume (for contradiction) that there exists some $\omega > t^2$ for which $\Wtransbar{n}{2\uu t} {p}(\omega) = 0$.
This would imply that
\[
 \Htransbar{n}{p}(\omega)
= \Htransbar{n}{p}(t^2)
= \frac{1}{4\uu^2t^2}.
\]
However this is impossible, since $\Htransbar{n}{p}(x)$ is strictly decreasing whenever $x > \maxroot{p}$ and we have assumed that $\omega > t^2 > \maxroot{p}$.
The reverse direction (once we have fixed the value of $t$) is essentially the same.
\end{proof}

Lemma~\ref{lem:recPinch} then follows easily:
\begin{proof}[Proof of Lemma~\ref{lem:recPinch}]
Apply Lemma~\ref{lem:pinch} with $\uuu = 2 \uu t$ to get polynomials $\ptil$ and $\phat$.
Using Lemma~\ref{lem:quad}, the properties provided by Lemma~\ref{lem:pinch} then translate directly into the properties needed.
\end{proof}

\subsection{The Lemmas}\label{sec:proofinduct}

We now give proofs of the two main lemmas stated at the beginning of the section (the lemmas are restated here for convenience).

\begin{replemma}{lem:inductiond}
Let $2 \leq \jj < \dd$.
Assume that, for some fixed $q \in \rrposle{\dd-1}$ and $\uu > 0$, we have
\begin{enumerate}
\item[(i)] $\myphi{\nn}{\kk}{\dd-1}{\uu}{p}{q} > 0$ for 
all $p \in \rrpos{\jj-1}$
\item[(ii)] $\myphi{\nn}{\kk}{\dd}{\uu}{p}{x^{\dd-1}} > 0$ for 
all $p \in \rrpos{\jj}$
\end{enumerate}
Then $\myphi{\nn}{\kk}{\dd}{\uu}{p}{q} > 0$ for 
all $p \in \rrpos{j}$.
\end{replemma}
\begin{proof}
Since both $p$ and $q$ have degree less than $\dd$, we can apply Lemma~\ref{lem:reduce} to get
\[
[ p \mysum{\dd}{\kk}  q ] = \frac{1}{\dd (\dd + \kk)} [ \Dp \mysum{\dd-1}{\kk}  q ]
\]
where $\Dp = [p \mysum{\dd}{\kk}  x^{\dd-1} ] \in \rrpos{\jj-1}$.
In particular, 
\[
\Theta^{\nn}_\uu([ p \mysum{\dd}{\kk}  q ]) 
= \Theta^{\nn}_\uu([ \Dp \mysum{\dd-1}{\kk}  q ])
\]
where by $(i)$ we have
\[
\Theta^{\nn}_\uu([ \Dp \mysum{\dd-1}{\kk}  q ]) < \Theta^\nn_\uu(\Dp) + \Theta^\nn_\uu(q) - (\nn + 2(\dd-1))\uu
\]
and by $(ii)$ we have
\[
\Theta^\nn_\uu(\Dp) < \Theta^\nn_\uu(p) + \Theta^\nn_\uu(x^{\dd-1}) - (\nn + 2 \dd) \uu.
\]
where $\Theta^\nn_\uu(x^{\dd-1}) = (\nn + 2 (\dd-1)) \uu$ by Lemma~\ref{lem:basictheta}.
Combining these gives
\[
\Theta^{\nn}_\uu([ p \mysum{\dd}{\kk}  q ]) < \Theta^\nn_\uu(p) + \Theta^\nn_\uu(q) - (\nn + 2\dd)\uu
\]
as required.
\end{proof}

\begin{replemma}{lem:induction}
Let $2 \leq \jj \leq \dd$.
Assume that, for a fixed $q \in \rrposle{\dd}$ and $\uu > 0$, we have
\begin{enumerate}
\item[(i)] $\myphi{\nn}{\kk}{\dd}{\uu}{p}{q} > 0$ for 
all polynomials $p \in \rrpos{\jj-1}$
\item[(ii)] $\myphi{\nn}{\kk}{\dd}{\uu}{p}{q} > 0$ for 
all basic polynomials $p \in \rrpos{\jj}$
\end{enumerate}
Then $\myphi{\nn}{\kk}{\dd}{\uu}{p}{q} > 0$ for 
all polynomials $p \in \rrpos{\jj}$.
\end{replemma}
\begin{proof}
Fix $\uu$ and $q$ and assume (for contradiction) that there exists a polynomial $p \in \rrpos{\jj}$ for which $\myphi{\nn}{\kk}{\dd}{\uu}{p}{q} \leq 0$.
Since $p$ has a finite number of roots, we can find a constant $R$ for which all of the roots of $p$ lie in the interval $[0, R]$, and then consider the collection of all polynomials in $\rrpos{\jj}$ whose roots lie in the interval $[0, R]$.
This collection is compact, and so $\myphi{\nn}{\kk}{\dd}{\uu}{\cdot}{q}$ (a continuous function in the roots of $p$) achieves its minimum.
Let $p_0$ be a polynomial achieving this minimum; in particular, note that  $\myphi{\nn}{\kk}{\dd}{\uu}{p_0}{q} \leq \myphi{\nn}{\kk}{\dd}{\uu}{p}{q} \leq 0$, which means (by hypothesis $(ii)$) $p_0$ must be nonbasic.

Since $p_0$ is nonbasic, it has a decomposition $p_0 = \ptil_0 + \phat_0$ with the properties of Lemma~\ref{lem:recPinch}.
In particular, note that 
\begin{enumerate}
\item[a.] implies (by the linearity of $\mysum{\dd}{\kk}$) that $[p_0 \mysum{\dd}{\kk} q] = [\ptil_0 \mysum{\dd}{\kk} q] + [\phat_0 \mysum{\dd}{\kk} q]$
\item[b.] implies that $\ptil_0$ has all of its roots in $[0, R]$, so the choice of $p_0$ as a minimizer ensures that $\myphi{\nn}{\kk}{\dd}{\uu}{p_0}{q} \leq \myphi{\nn}{\kk}{\dd}{\uu}{\ptil_0}{q}$.
\item[c.] implies (by plugging in to the definition of $\myphi{\nn}{\kk}{\dd}{\uu}{\cdot}{\cdot}$)
\begin{align*}
  \myphi{\nn}{\kk}{\dd}{\uu}{p_0}{q} + \Theta^{\nn}_{\uu}(p_0 \mysum{\dd}{\kk} q) 
= \myphi{\nn}{\kk}{\dd}{\uu}{\ptil_0}{q} + \Theta^{\nn}_{\uu}(\ptil_0 \mysum{\dd}{\kk} q)
= \myphi{\nn}{\kk}{\dd}{\uu}{\phat_0}{q} + \Theta^{\nn}_{\uu}(\phat_0 \mysum{\dd}{\kk} q).
\end{align*}
\end{enumerate}
Hence Properties b. and c. combine to give $\Theta^{\nn}_{\uu}(\ptil_0 \mysum{\dd}{\kk} q) \leq \Theta^{\nn}_{\uu}(p_0 \mysum{\dd}{\kk} q)$ and so 
\begin{equation}\label{eq:c1}
\maxroot{\Wtrans{\nn}{\uu}{[\ptil_0 \mysum{\dd}{\kk} q]}}
\leq
\maxroot{\Wtrans{\nn}{\uu}{[p_0 \mysum{\dd}{\kk} q]}}.
\end{equation}

Now let $\beta = \maxroot{\Wtrans{\nn}{\uu}{[p_0 \mysum{\dd}{\kk} q]}}$ and assume (for the moment) that 
\begin{equation}\label{eq:doit}
\beta^2 \geq \maxroot{[p_0 \mysum{\dd}{\kk} q] \times [\ptil_0 \mysum{\dd}{\kk} q] \times [\phat_0 \mysum{\dd}{\kk} q]}.
\end{equation}
This would imply that the decomposition 
\[
[p_0 \mysum{\dd}{\kk} q] = [\ptil_0 \mysum{\dd}{\kk} q] + [\phat_0 \mysum{\dd}{\kk} q]
\]
satisfies the requirements of Corollary~\ref{cor:straddle_roots}, which would allow us to extend (\ref{eq:c1}) to
\[
\maxroot{\Wtrans{\nn}{\uu}{[\ptil_0 \mysum{\dd}{\kk} q]}}
\leq
\maxroot{\Wtrans{\nn}{\uu}{[p_0 \mysum{\dd}{\kk} q]}}
\leq 
\maxroot{\Wtrans{\nn}{\uu}{[\phat_0 \mysum{\dd}{\kk} q]}}.
\]
Using Property c. again, this would then imply that $\myphi{\nn}{\kk}{\dd}{\uu}{\phat}{q} \leq \myphi{\nn}{\kk}{\dd}{\uu}{p_0}{q} < 0$, a contradiction to the initial hypothesis (since $\phat \in \rrpos{j-1}$).

Hence it suffices to prove (\ref{eq:doit}).
First we note that Corollary~\ref{lem:monotU} easily gives
\[
\beta^2 
= \maxroot{\Wtrans{\nn}{\uu}{[p_0 \mysum{\dd}{\kk} q]}}^2
\geq \maxroot{\Wtrans{\nn}{0}{[p_0 \mysum{\dd}{\kk} q]}}^2
= \maxroot{[p_0 \mysum{\dd}{\kk} q]}
\] 
and that (\ref{eq:c1}) combined with the same Corollary gives
\[
\beta^2 
\geq \maxroot{\Wtrans{\nn}{\uu}{[\ptil_0 \mysum{\dd}{\kk} q]}}^2
\geq \maxroot{\Wtrans{\nn}{0}{[\ptil_0 \mysum{\dd}{\kk} q]}}^2
= \maxroot{[\ptil_0 \mysum{\dd}{\kk} q]}.
\]
Finally the fact that $\beta^2 \geq \maxroot{[\phat_0 \mysum{\dd}{\kk} q]}$ comes directly from Property d. in Theorem~\ref{lem:recPinch}.
\end{proof}


\section{Base Cases}\label{sec:proof}

Our proof of Theorem~\ref{thm:main} is inductive and utilizes both 
Lemma~\ref{lem:inductiond} and Lemma~\ref{lem:induction}.
Each of these lemmas will require its own ``base case''.

\begin{enumerate}
\item $p \in \rrpos{d}$ and $q(x) = x^{d-1}$ (Corollary~\ref{cor:allderiv}), and
\item $p, q \in \rrpos{d}$ are both basic polynomials (Lemma~\ref{lem:geg0}).
\end{enumerate}

Neither of these lemmas is particularly simple (despite being the ``base 
cases'' of an induction).
We prove Corollary~\ref{cor:allderiv} using a separate induction, first 
considering the case when $p$ is a basic polynomial (Lemma~\ref{lem:deriv}) and 
then using Lemma~\ref{lem:induction} to extend this to all polynomials.
While the proof of Lemma~\ref{lem:deriv} is mostly calculus, the functions that 
one needs to consider become complicated enough that we were forced to appeal 
to the aid of a computer in order to calculate them.
Section~\ref{sec:derivative} is dedicated to this case.

Lemma~\ref{lem:geg0}, on the other hand, will require an entire investigation 
of its own.
Corollary~\ref{cor:rectToGeg} will relate the rectangular additive convolution 
of two basic polynomials a well-studied class of orthogonal polynomials and we 
will utilize a number of well-known properties of these polynomials to prove 
the necessary inequalities.
To simplify the presentation, the proof of Lemma~\ref{lem:geg0} given in 
Section~\ref{sec:basic} will be contingent upon a bound on the Cauchy transform 
of Gegenbauer polynomials that will be proved in Section~\ref{sec:geg}.

Assuming these two base cases, the proof of Theorem~\ref{thm:main} is then 
straightforward:

\begin{reptheorem}{thm:main}
For $n\geq 0, \uu > 0, \dd \geq 1$ and $p, q \in \rrposle{d}$, we have 
$\myphi{\nn}{\kk}{\dd}{\uu}{p}{q} > 0$.
\end{reptheorem}
\begin{proof}
We proceed by induction on $\dd$.
The base case (when $\dd = 1$) is covered by Lemma~\ref{lem:geg0}, since degree 
$1$ polynomials are (by definition) basic.

Now assume that, for some $D \geq 2$ the theorem holds for all $\dd < D$ and 
consider $p, q \in \rrposle{D}$.
If $\deg(p) < D$ or $\deg(q) < D$, we can appeal to the inductive hypothesis 
using Lemma~\ref{lem:inductiond} and Corollary~\ref{cor:allderiv}.
On the other hand, if $\deg(p) = \deg(q) = D$, we can appeal to the inductive 
hypothesis using Lemma~\ref{lem:induction} and Lemma~\ref{lem:geg0}.
\end{proof}

In order to make the presentation of the two lemmas more readable, we first show that the parameter $\alpha$ can be effectively scaled out of statements pertaining to $\myphi{\nn}{\kk}{\dd}{\uu}{p}{q}$.
Note that the transformation $p \to p_\alpha$ in Lemma~\ref{lem:alpha1} preserves basic polynomials, so statements that are restricted to basic polynomials can be scaled out as well.

\begin{lemma}\label{lem:alpha1}
For a fixed $\alpha > 0$ and polynomials $p, q$ with nonnegative roots, let $p_\alpha(x) = p(\alpha^2 x)$ and $q_\alpha(x) = q(\alpha^2 x)$ .
Then 
\[
\myphi{\nn}{\kk}{\dd}{\uu}{p}{q} = \alpha \myphi{\nn}{\kk}{\dd}{1}{p_\alpha}{q_\alpha}
\]
\end{lemma}
\begin{proof}
By (\ref{def:myphi}), it suffices to show
\[
\maxroot{ \Wtrans{\nn}{\uu}{p}}
 = 
\alpha \cdot \maxroot{\Wtrans{\nn}{1}{p_\alpha}}
\]
for arbitrary $p$.
Computing, we have
\begin{align*}
\Wtrans{\nn}{1}{p_\alpha} (x) 
&= x^{2n} p_\alpha(x^2)^2 - 1[x^{2n} p_\alpha(x^2)]'[r(x^2)]'
\\&= x^{2n} p(\alpha^2 x^2)^2 - \left( 2n x^{2n-1} p(\alpha^2x^2) + 2 x^{2n+1} \alpha^2 p'(\alpha^2 x^2) \right) \left(2\alpha^2 x p'(\alpha^2x^2) \right)
\\&= x^{2n} \left( p(\alpha^2 x^2)^2 - 4 n \alpha^2 p(\alpha^2x^2)p'(\alpha^2x^2) - 4 x^2 \alpha^4 p'(\alpha^2 x^2)^2 \right)
\end{align*}
so 
\begin{align*}
\maxroot{\Wtrans{\nn}{1}{p_\alpha}}
&=
\maxroot{p(\alpha^2 x^2)^2 - 4 n \alpha^2 p(\alpha^2x^2)p'(\alpha^2x^2) - 4 x^2 \alpha^4 p'(\alpha^2 x^2)^2 }
\\&= \frac{1}{\alpha} \cdot \maxroot{p(x^2)^2 - 4 n \alpha^2 p(x^2)p'(x^2) - 4 x^2 \alpha^2 p'(x^2)^2 }
\\&= \frac{1}{\alpha} \cdot \maxroot{\Wtrans{\nn}{\uu}{p}}.
\end{align*}
\end{proof}

\subsection{The case of $q(x) = x^{d-1}$}\label{sec:derivative}

As mentioned earlier, we start by proving the case when $p$ is a basic polynomial.

\begin{lemma}\label{lem:deriv}
Let $\dd \geq 2$, $\nn \geq 0$, $\uu > 0$ and let $p \in \rrpos{\dd}$ be a basic polynomial.
Then $\myphi{\nn}{\nn}{\dd}{\uu}{p}{x^{\dd-1}} > 0$.
\end{lemma}
\begin{proof}
By Lemma~\ref{lem:alpha1}, it suffices to prove the lemma when $\uu = 1$.
To simplify notation slightly, we will sometimes write $m = n + d$.
Note that for $p(x) = (x - \lambda)^d$ and $q(x) = x^{d-1}$, Lemma~\ref{lem:basictheta} implies that 
\[
\Theta^{n}_1(p) = \sqrt{m^2 + \lambda} + d
\AND
\Theta^{n}_1(q) = n - 2.
\]
Letting 
\[
\Dp 
= [ p \mysum{d}{n} x^{d-1}]
= (x-\lambda)^{d-2} \left(x- \frac{n+1}{m} \lambda \right)
\]
it then suffices to show that
\[
\sqrt{n^2+\maxroot{\Wtrans{n}{1}{\Dp}}^2} < \sqrt{m^2 + \lambda} + d - 2
\]
or, squaring both sides and rearranging,
\begin{equation} \label{eq:az}
\maxroot{\Wtrans{n}{1}{\Dp}}^2 < \bigg(2(d-1)-m+\sqrt{m^2+\lambda}\bigg)\bigg(m-2+\sqrt{m^2+\lambda})\bigg) := \mu_\lambda
\end{equation}

Furthermore, we can calculate
\[
\Wtrans{n}{1}{\Dp(x)}
=(x^2-\lambda)^{2(d-3)}x^{2 n} s_{\lambda}(x^2)
\]
where $s_\lambda \in \rrpos{4}$ has coefficients that depend on $n, d, \lambda$.
Hence 
\[
\maxroot{\Wtrans{n}{1}{\Dp(x)}}^2 = \maxroot{s_{\lambda}},
\]
and so (\ref{eq:az}) is equivalent to having $\maxroot{s_{\lambda}} < \mu_\lambda$.
Our approach to proving this will be to show that $s_{\lambda}(\mu_{\lambda} + \epsilon) > 0$ for all $\epsilon \geq 0$.
Consider the Taylor series of the function
\[
f(\lambda)= s_{\lambda}(\mu_{\lambda}).
\]
around $\lambda = 0$.
With the help of a computer, one can calculate that $f(0) = f'(0) = 0$ and that
\[
f^{(2)}(0)= \frac{32(d-1)^2(m-1)(n+1)(m+d-2)}{m^3} > 0
\]
and
\[
 f^{(3)}(\lambda)= \frac{24(n+1)(d-2) g(\lambda)}{m^2(m^2+\lambda)^{\frac{5}{2}}}
\]
where
\[
g(\lambda)= 2(d-1)(m^4-m^3)+ m^2\bigg((2d-1)^2+3\lambda d-4\lambda \bigg) +\frac{5}{4}(d-1)\lambda^2+(d-1)m \lambda.
\]
So for $d \geq 2$, we have $(2d-1)^2+3\lambda d-4\lambda \geq 0$, and for $m \geq d$, all other terms in $f^{(3)}(\lambda)$ are trivially nonnegative.
Hence $f(\lambda) > 0$ for all $\lambda > 0$.

Now note that, because the leading coefficient of $s_\lambda$ is positive, the fact that $s_\lambda(\mu_\lambda) > 0$ implies that the number of roots of $s_\lambda$ that are larger than $\mu_\lambda$ is even.
On the other hand, when $\lambda = 0$, 
\[
s_0(x) = x^3(x - 4(m-1)(d-1))
\]
and so has three roots at $x = 0$ and one at $\mu_0 = 4(m-1)(d-1) > 0$.
As $s_{\lambda }$ is real rooted for all $\lambda \geq 0$ and the roots of $s_{\lambda}$ are continuous functions of its coefficients (and thus of $\lambda$) we can conclude that for small $\lambda$ all but one of the roots of $s_{\lambda}$ must be near $0$; 
in other words, the function 
\[
g(\lambda) = \maxroot{s_\lambda} - \mu_\lambda
\]
is positive for sufficiently small $\lambda$.
Hence it suffices to show that $g(\lambda) \neq 0$ for any $\lambda > 0$.

Assume (for contradiction) that there exists $\lambda_0 > 0$ for which $g(\lambda_0) = 0$.
In other words, $\maxroot{s_{\lambda_0}} = \mu_{\lambda_0}$ which (in particular) means that $s_{\lambda_0}(\mu_{\lambda_0}) = 0$.
But this is a contradiction, since we have shown $f(\lambda) = s_{\lambda}(\mu_{\lambda})$ to be strictly positive for $\lambda > 0$.
Hence $g(\lambda)$ must remain strictly positive for all $\lambda > 0$, finishing the proof.
\end{proof}

Using Lemma~\ref{lem:induction}, we can then extend Lemma~\ref{lem:deriv} to all polynomials.

\begin{corollary}\label{cor:allderiv}
Let $2 \leq \jj \leq \dd$, $\nn \geq 0$, $\uu > 0$ and let $p \in \rrpos{\jj}$ be an (arbitrary) polynomial.
Then $\myphi{\nn}{\nn}{\dd}{\uu}{p}{x^{\dd-1}} > 0$.
\end{corollary}

\begin{proof}
Let $S(\jj, \dd)$ be the statement
\[
\myphi{\nn}{\nn}{\dd}{\uu}{p}{x^{\dd-1}} > 0 \text{ for all $n \geq 0$, $\uu > 0$ and all $p \in \rrpos{\jj}$}.
\]
Our goal is then to prove $S(\jj, \dd)$ for all $2 \leq \jj \leq \dd$, and we 
will proceed by induction on $\jj + \dd$.
The base case is when $\jj = \dd = 2$, which we will consider below.
Now assume $S(\jj, \dd)$ to hold whenever $\jj + \dd < K$ and consider a polynomial $p$ with degree $\jj \leq \dd$ for which $\jj + \dd = K$.
We split into two cases:
\begin{enumerate}
\item~$\jj < \dd$:

Since both $p$ and $x^{d-1}$ have degree less than $\dd$, we can apply Lemma~\ref{lem:reduce} to get
\[
[ p \mysum{\dd}{\nn}  x^{d-1} ] 
= \frac{1}{\dd (\dd + \nn)} [ p \mysum{\dd-1}{\nn} [x^{d-1} \mysum{\dd}{\nn} x^{\dd-1} ] ]
= c_{\dd, \nn} [ p \mysum{\dd-1}{\nn} x^{\dd-2} ]
\]
for some constant $c_{\dd, \nn}$.
That is, $S(\jj, \dd-1) \Rightarrow S(\jj, \dd)$.

\item~$\jj = \dd$:

A combination of Lemma~\ref{lem:induction} and Lemma~\ref{lem:deriv} gives that
\[ 
S(\dd-1, \dd) \Rightarrow S(\dd, \dd) 
\]
\end{enumerate}
Hence in both cases we have reduced the statement $S(\jj, \dd)$ to one that we know (by the inductive hypothesis) is true, so $S(\jj, \dd)$ is true as well.

To finish the proof, we then need to consider only the base case: $\jj = \dd = 2$.
For $t \in [0, a]$, set
\[
p_t(x) = (x-a)^2 - t^2 = (x - a - t)(x - a + t)
\]
and note that
\[
[ p_t  \mysum{2}{n} x] = x - a \frac{n+1}{n+2}
\]
is independent of $t$.
Hence we have
\[
\myphi{\nn}{\nn}{1}{\uu}{p_t}{x} 
 = \Theta^{n}_{\uu}(p_t) - \Theta^{n}_{\uu}(p_0) + \myphi{\nn}{\nn}{1}{\uu}{p_0}{x}
\]
where $\myphi{\nn}{\nn}{1}{\uu}{p_0}{x} > 0$ by Lemma~\ref{lem:deriv}, and so it suffices to show that 
\begin{equation}\label{eq:theta_t}
\maxroot{\Wtrans{\nn}{\uu}{p_t}} \geq \maxroot{\Wtrans{\nn}{\uu}{p_0}}.
\end{equation}
By Corollary~\ref{lem:translate} and Lemma~\ref{lem:simplify}, this is equivalent to showing that 
\[
[ \log p_t ]'(x) \leq [ \log p_0 ]'(x).
\]
for all $x > \maxroot{p_0 p_t} = a+t$.
However, it is easy to check that
\[
\frac{\partial}{\partial t} [ \log p_t ]'(x) = \frac{4 t (a - x)}{(x - a - t)^2(x - a + t)^2},
\]
which is negative for $x > a + t$. 
\end{proof}

\subsection{The case of basic polynomials}\label{sec:basic}

We start by finding a generating function for the rectangular convolution of two basic polynomials.
The derivation will use the following well known generalization of the {\em binomial theorem} (see, for example, \cite{Wilf}).

\begin{theorem}\label{thm:binom}
The function $(1+z)^{-k}$ has the formal power series expansion
\[
\frac{1}{(1+z)^k} 
= \sum_{i=0}^{\infty} \binom{k+i-1}{i}(-z)^i.
\]
\end{theorem}

\begin{lemma}\label{lem:gf}
For all $\lambda, \mu > 0$ and $\kk \geq 0$, the polynomials 
\[
q^{\lambda, \mu}_{\kk, \dd}(y) = 
\binom{\kk + \dd}{\dd} [ (x-\lambda)^\dd \mysum{\dd}{\kk} (x-\mu)^\dd ](y).
\]
satisfy the formal power series identity
\[
\sum_{\dd} q^{\lambda, \mu}_{\kk, \dd}(y)~t^\dd
= \frac{1}{((1 + \mu t)(1 + \lambda t)  - y t )^{\kk+1}}.
\]
\end{lemma}

\begin{proof}
Using (\ref{eq:alt_def}), we can write $q^{\lambda, \mu}_{\kk, \dd}$ explicitly.
Letting 
\[
c(s, i, j, l) = \frac{(s-i)!(s-j)!}{s!(s-\ell)!}
\]
we have
\begin{align*}
q^{\lambda, \mu}_{\kk, \dd}(y)
&= \binom{\kk + \dd}{\dd} \sum_{\ell} y^{\dd-\ell}
\left( 
\sum_{i+j=\ell}
c(\dd, i, j, \ell) 
c(\kk+\dd,i, j, \ell) 
\binom{\dd}{i}\binom{\dd}{j} (-\mu)^i (-\lambda)^j 
\right)
\\&= \sum_{i+j+\ell = \dd} \binom{\kk +\ell}{\ell}\binom{\kk + j+ 
\ell}{j}\binom{\kk + i + \ell}{i} (-\mu)^i (-\lambda)^j y^{\ell}.
\end{align*}
Hence we have the formal power series identity
\begin{align*}
\sum_\dd q^{\lambda, \mu}_{\kk, \dd}(y)~t^\dd
&= 
 \sum_{i, j, \ell} \binom{\kk +\ell}{\ell}\binom{\kk + j+ \ell}{j}\binom{\kk + 
 i + \ell}{i} (-\mu)^i (-\lambda)^j y^{\ell} t^{i + j + \ell} 
\\&=  \sum_\ell (1 + \mu t)^{-(\kk + \ell + 1)}(1 + \lambda t)^{-(\kk + \ell + 
1)} \binom{\kk +\ell}{\ell} y^{\ell} t^{\ell} \label{eq:a1} \qquad 
\text{(Theorem~\ref{thm:binom} on $i, j$) }
\\&= \big((1 + \mu t)(1 + \lambda t)\big)^{-(\kk + 1)}\sum_{\ell} \binom{\kk 
+\ell}{\ell} \left(\frac{ y t}{(1 +\mu t)(1 + \lambda t)}\right)^{\ell} 
\\&= ((1 + \mu t)(1 + \lambda t)  - y t )^{-(\kk+1)} \hspace{4.5cm}
\text{(Theorem~\ref{thm:binom} on $\ell$)}.
\end{align*}
\end{proof}

This provides an easy link between the rectangular convolution of basic polynomials and the Gegenbauer polynomials studied in Section~\ref{sec:geg}:

\begin{corollary}\label{cor:rectToGeg}
For all $\lambda, \mu > 0, \kk \geq 0, \dd \geq 1$ we have
\[
\binom{\kk + \dd}{\dd} [ (x-\lambda)^\dd \mysum{\dd}{\kk} (x-\mu)^\dd ](y)
= 
 (\lambda \mu)^{\dd/2}
C_{\dd}^{\nn+1} \left(\frac{y - (\lambda +\mu)}{2 \sqrt{\lambda \mu }} \right).
\]
\end{corollary}
\begin{proof}
Compare Lemma~\ref{lem:gf} to (\ref{eq:gf}).
\end{proof}

Bounding the $H$-transform of the convolution of two basic polynomials can therefore be reduced to bounding the Cauchy transform of Gegenbauer polynomials.
We prove the necessary bound in Section~\ref{sec:geg}, but reproduce the theorem here for continuity.

\begin{reptheorem}{thm:ans}
Consider the bivariate polynomial 
\[
f(x, y) = \dd (x^2 -1)y^2 + 2 \nn x y - (2\nn + d)
\]
and assume that for a given $s > \gamma_{\nn/\dd}$ and $t > 0$ that $f(s, t) \geq 0$.
Then 
\[
\cauchy{C_{\dd}^{\nn+1}}{s} = \frac{[C_{\dd}^{\nn+1}]'(s)}{d [C_{\dd}^{\nn+1}](s)]} \leq t.
\]
\end{reptheorem}

The proof of Lemma~\ref{lem:geg0} will require a number of identities that are not, by themselves, important to understanding the overall proof.
In order to keep the continuity of ideas in the proof, we have separated these identities out into a separate lemma:

\begin{lemma} \label{lem:claimT}
For $\lambda, \mu, \nn, \dd \in \R$, let
\[
t := \sqrt{(\nn+\dd)^{2} + \lambda} + \sqrt{(\nn+\dd)^{2} + \mu}
\AND
t_* :=\sqrt{t^2-2t\nn}
\]
and also let
\[
T := \frac{{t_*}^2-(\lambda+\mu)}{2\sqrt{\lambda \mu}}
\AND
R := \frac{\sqrt{\lambda\mu}}{\dd t}
\AND
\gamma_{\nn/\dd} := \sqrt{1 - \frac{\nn^2}{(\nn+\dd)^2}}.
\]
Then the following identities hold:
\begin{enumerate}
\item $\dd (T^2 -1)R^2 + 2 \nn T R - (2\nn + d) = 0$.
\item $\left(\dd R (T^2-1)+ \nn T \right)^2 = (\nn + \dd)^2 (T^2- \gamma_{\nn/\dd}^2)$
\end{enumerate}
\end{lemma}
\begin{proof}
For real numbers $a, b$, it is easy to check that $\sqrt{a} + \sqrt{b}$ is a root of the polynomial $x^4 - 2(a + b)x^2 + (a-b)^2$ by simple substitution.
Hence $t^4- 2\left(\lambda + \mu + 2(\nn+\dd)^2 \right) t^2 + (\lambda - \mu)^2=0$, or (rearranging slightly), 
\begin{equation}\label{eq:Tpoly}
\left(t^2 - (\lambda + \mu) \right)^2 
= 4 (\nn+\dd)^2 t^2 + 4 \lambda \mu
= 4 t^2 \left( (\nn+\dd)^2 + \dd^2 R^2 \right)
\end{equation}
By definition of $T$, however, we have
\begin{equation}\label{eq:Tdef}
\left(t^2 - (\lambda+\mu) \right)^2
= \left(2\sqrt{\lambda \mu} T + 2 t \nn \right)^2
= \left(2\dd t R T + 2 t \nn \right)^2
= 4 t^2 \left( \dd^2 R^2 T^2 + 2 \dd \nn R T + \nn^2 \right)
\end{equation}
Equating (\ref{eq:Tpoly}) and (\ref{eq:Tdef}) and then rearranging gives
\begin{equation}\label{eq:T1}
\dd^2 R^2 (T^2 - 1) + 2 \nn \dd R T = \left( (\nn+\dd)^2 - \nn^2 \right) = 2 \nn \dd + \dd^2,
\end{equation}
which clearly implies {\em 1.}
Now note that if we multiply out {\em 2.}, we get
\begin{align*}
\dd^2 R^2 (T^2-1)^2 + 2 \dd R (T^2-1) \nn T + \nn^2 T^2
&= (\nn + \dd)^2 \left( T^2 - \left(1 - \frac{\nn^2}{(\nn+\dd)^2} \right) \right)
\\&= (\nn + \dd)^2 (T^2 -1) + \nn^2 T^2
\end{align*}
which, after canceling the $\nn^2 T^2$ terms and dividing out $T^2-1$ matches (\ref{eq:T1}).
\end{proof}

Finally, we are able to prove the lemma:

\begin{lemma}\label{lem:geg0}
Let $\dd \geq 1$, $\nn \geq 0$, $\uu > 0$ and let $p, q \in \rrpos{\dd}$ be basic polynomials.
Then $\myphi{\nn}{\nn}{\dd}{\uu}{p}{q} > 0$.
\end{lemma}
\begin{proof}
Again by Lemma~\ref{lem:alpha1} it suffices to consider the case $\uu=1$.
We set
\[
p(x) = (x-\aa)^d
\AND
q(x) = (x-\bb)^d
\AND
r(x) = [p \mysum{\dd}{\kk} q](x).
\]
and note that Corollary~\ref{cor:rectToGeg} and Theorem~\ref{thm:ans} show that 
\begin{equation}\label{eq:maxrootr}
\frac{\maxroot{r} - (\lambda + \mu)}{2\sqrt{\lambda \mu}} 
\leq \gamma_{\nn/\dd} 
= \sqrt{1 - \frac{\nn^2}{(\nn+\dd)^2}}.
\end{equation}
Also set $w_r = \maxroot{\Wtrans{\nn}{1}{r}}$; substituting this and Lemma~\ref{lem:basictheta} into the definition of $\myphi{\nn}{\nn}{\dd}{1}{\cdot}{\cdot}$, we get that it suffices to show
\begin{equation}\label{eq:ineq}
\sqrt{\nn^2  + w_r^2} \leq \sqrt{(\nn + \dd)^2  + \lambda} + \sqrt{(\nn + \dd)^2  + \mu} - \nn.
\end{equation}
Next define the quantities
\begin{align*}
t 
:= \sqrt{(\nn+\dd)^{2} + \lambda} + \sqrt{(\nn+\dd)^{2} + \mu} 
\AND
t_*:=\sqrt{t^2-2t\nn},
\end{align*}
noting that the string of inferences
\[
\lambda, \mu > 0 
\Rightarrow 
t \geq 2(\nn + \dd)
\Rightarrow
t(t-2\nn) \geq 2 \dd t \geq 0
\]
show that $t^*$ is well defined.
Then to prove (\ref{eq:ineq}), it suffices to show (after squaring both sides, and substituting) the inequality $w_r \leq t_*$.

Finally, define 
\[
W := \frac{w_r^2 - \lambda - \mu}{2 \sqrt{\lambda \mu}}
\AND
T := \frac{t_*^2 - \lambda - \mu}{2 \sqrt{\lambda \mu}} 
\]
and notice that the identity in Lemma~\ref{lem:claimT}.2 implies that $T \geq \gamma_{\nn/\dd}$.
Then if $W \leq \gamma_{n/d}$, we are done, so we are left to consider the case when $W \geq \gamma_{n/d}$.

For $W, T \geq \gamma_{n/d}$, we then have by (\ref{eq:maxrootr}) that 
\[
w_r \leq \sqrt{\maxroot{r}}
\AND
t_* \leq \sqrt{\maxroot{r}}
\]
so Lemma~\ref{lem:Hmonotone} implies that the inequality $w_r \leq t_*$ holds if and only if
\begin{equation}\label{eq:Ht*}
\Htrans{n}{r}(t_*) \leq \Htrans{n}{r}(w_r) = 1.
\end{equation}
By definition, we have
\begin{align*}
\Htrans{n}{r}(x)
= 2 \dd x \cauchy{r}{x^2} \left( \frac{2 \nn}{x} + 2 \dd x \cauchy{r}{x^2} \right)
= 4 \dd \left( \nn \cauchy{r}{x^2} + \dd x^2 \cauchy{r}{x^2}^2 \right)
\end{align*}
and so we can rewrite (\ref{eq:Ht*}) as
\[
 \dd t_*^2 \cauchy{r}{t_*^2}^2 + \nn \cauchy{r}{t_*^2} \leq \frac{1}{4\dd}.
\]
We now bound $\cauchy{r}{t_*^2}$ by noting that the identity in Lemma~\ref{lem:claimT}.1 allows us to apply Theorem~\ref{thm:ans} to the ordered pair
\[
(s, t) = \left(\cauchy{\geg{\dd}{\nn+1}}{T}, \frac{\sqrt{\lambda \mu}}{\dd t} \right).
\]
By Corollary~\ref{cor:rectToGeg}, we then get the inequality
\[
2 \sqrt{\lambda \mu} \cauchy{r}{t_*^2} 
= \cauchy{\geg{\dd}{\nn+1}}{T}
\leq \frac{\sqrt{\lambda \mu}}{ \dd t }
\quad
\Longrightarrow
\quad
\cauchy{r}{t_*^2} 
\leq \frac{1}{ 2\dd t }.
\]
Hence
\[
\dd t_*^2 \cauchy{r}{t_*^2}^2  + \nn \cauchy{r}{t_*^2} 
\leq \frac{t_*^2}{4 \dd t^2} + \frac{\nn}{2 \dd t}
= 
\frac{t_*^2 + 2 \nn t}{4 \dd t^2} = \frac{1}{4\dd}
\]
as required.
\end{proof}


\section{Gegenbauer polynomials}\label{sec:geg}

The Gegenbauer (or ultraspherical) polynomials $\geg{n}{\alpha}(x)$ are a collection of polynomials which are orthogonal with respect to the weight function $w(x) = (1 - x^2)^{\alpha - 1/2}$ on the interval $[-1, 1]$.
They can be computed explicitly using a generating function
\begin{equation}\label{eq:gf}
\sum_n \geg{n}{\alpha}(x) t^n = \frac{1}{(1 - 2 x t + t^2)^{\alpha}}
\end{equation}
or by the three-term recurrence given in Lemma~\ref{lem:identities}.
They are a special case of the more general {\em Jacobi polynomials}:
\[
\geg{d}{\alpha}(x)
= \frac{\Gamma(\alpha+1/2)\Gamma(d+2\alpha)}{\Gamma(2\alpha)\Gamma(d+\alpha + 1/2)} P_{d}^{(\alpha -1/2,\alpha -1/2)}(x)
\]
and are themselves a generalization of two other well-studied families of orthogonal polynomials:
\begin{enumerate}
\item Chebyshev polynomials of the second kind ($\alpha = 1$), and
\item Legendre polynomials ($\alpha = 1/2$).
\end{enumerate}

We will use the following identities, which can be derived directly from (\ref{eq:gf}) or taken as a specialization of known identities for Jacobi polynomials (see, for example, \cite{szego_book}).
Note that the $\Gamma(\cdot)$ appearing in the second identity is the usual {\em Gamma function}
\[
\Gamma(x) = \int_{0}^{\infty} t^{x-1}e^{-t} \d{t}.
\]

\begin{lemma}\label{lem:identities}
The following hold for all real numbers $\alpha > 0$ and all integers $d > 0$.
\begin{enumerate}
\item Recurrence relation, with convexity coefficients: 
\[
x\geg{d}{\alpha}(x)=\frac{d+1}{2d+2\alpha} \geg{d+1}{\alpha}(x) + \frac{d+2\alpha-1}{2d+2\alpha}\geg{d-1}{\alpha}(x)\]
\item Value(s) at 1:
 \[
\geg{d}{\alpha}(1)= \binom{d+2\alpha-1}{d}
= \frac{\Gamma(d+2\alpha)}{\Gamma(2\alpha)\Gamma(d+1)}
\AND
\cauchy{\geg{d}{\alpha}}{1}
= \frac{d + 2 \alpha}{2 \alpha +1}
\]
\item Differential equation:
\[
(1-x^2) \deriv{x} \geg{d}{\alpha}(x)
= -dx\geg{d}{\alpha}(x) +(d+2\alpha-1)\geg{d-1}{\alpha}(x)
\]
\end{enumerate}

\end{lemma}

Our goal will be to derive bounds on the Cauchy transform of $\geg{d}{\alpha}$ at values greater than its roots. 
Fortunately, much is known about the largest roots of Gegenbauer polynomials; in particular, the following bound from \cite{root_bound} provides a useful starting point:
\begin{lemma}[Elbert, Laforgia] \label{laforgia}
For fixed $\alpha > 0$, the roots of $\geg{d}{\alpha}$ lie in the interval $[-\gamma_{\alpha/d}, \gamma_{\alpha/d}]$ where 
\[
\gamma_\theta := \frac{\sqrt{2 \theta + 1}} {\theta+1}
\]
\end{lemma}
\begin{remark}\label{rem:gammatheta}
Note that $\gamma_\theta$ is decreasing in $\theta$, and so decreasing the degree $d$ or increasing the parameter $\alpha$ will cause the interval to shrink.
This coincides with the well known fact that all of the positive roots of $\geg{d}{\alpha}$ are increasing as $d$ grows and decreasing as $\alpha$ grows \cite{szego_book}.
\end{remark}

The bound we will derive differs from Lemma~\ref{laforgia} (and all other known bounds, as far as we can tell) in that it requires us to compare the largest roots of $\geg{\dd}{\alpha}$ as $\dd$ and $\alpha$ grow together in a linear way.
Specifically, we will be interested in understanding the largest root of 
${\geg{d}{1+d\theta}}$ for a generic constant $\theta$.
The intuition above tells us that the growth in $d$ and the growth in $\alpha = 1 + d \theta$ should push the roots in opposite directions, and so it should not be surprising that the computations become somewhat delicate. 

It should be noted that the asymptotic behavior of these polynomials has been studied, and it is well known that the density of the roots of these polynomials converges to a fixed distribution that will depend on the parameter $\theta$. 
The distribution can be computed using a result of Kuijlaars and Van Assche regarding the asymptotic root distributions of Jacobi polynomials \cite{asymptotic_zeroes}.
The following is a reformulated version of the theorem specific to our polynomials: 

\begin{theorem}[Kuijlaars, Van Assche]\label{thm:asymptgegen}
For all $\theta > 0$, the density of the roots of $\geg{d}{1 + \theta d}$ converges (as $d \to \infty$) to the function
\[
\mu_\theta(y) = \frac{1+\theta}{\pi} \frac{\sqrt{\gamma_\theta^2 - y^2}}{1-y^2} 
\one_{[-\gamma_\theta, \gamma_\theta]}(y).
\]
\end{theorem}

An asymptotic bound on the Cauchy transform can then be derived from Theorem~\ref{thm:asymptgegen}: \begin{corollary}\label{cor:asymptcauchygegen}
For all $\theta > 0$, and all $x > \gamma_\theta$, we have
\begin{equation}\label{eq:asympt}
\lim_{d \to \infty} \cauchy{\geg{d}{1 + \theta d}}{x} =\frac{-\theta x+(1+\theta)\sqrt{x^2-\gamma_\theta^2}}{(x^2-1)}. 
\end{equation}
\end{corollary}

We give a brief sketch of the computation in an Appendix for the benefit of the 
reader. 
Note that for the specific case of $\theta = 0$, the formula in Theorem~\ref{thm:asymptgegen} simplifies greatly to a well-known bound on the Cauchy transform of Chebyshev polynomials:
\[
\lim_{d \to \infty} \cauchy{\geg{d}{1}}{x}= \frac{1}{\sqrt{x^2-1}}
\]
for all $x > 1$.

The rest of this section will be devoted to showing that the bound in Corollary~\ref{cor:asymptcauchygegen} holds for the individual polynomials $\geg{d}{1 + \theta d}$ as well.
In particular, we will show that the the sequence $\{ \cauchy{\geg{d}{1 + \theta d}}{x} \}_d$ is increasing in $d$ for all $\theta > 0$ and all $x >\gamma^d_{\theta}$ where we define
\[
\gamma^d_{\theta} := \maxroot{\geg{d}{1+\theta d}(x)}.
\]
One part of this will be to show that the sequence $\{ \gamma^d_{\theta} \}$ is an  increasing function of $d$ (for fixed $\theta$).
For the moment, however, we will find it convenient to consider the following ``two-step'' maximum root: 
\[
\Gamma^d_{\theta}:= \max \left\{ \gamma^d_{\theta},  \gamma^{(d+1)}_{\theta} \right\}.
\]

\subsection{Nonasymptotic bounds}

To simplify notation slightly, we will fix $\theta$ and normalize the polynomials of interest (in a manner that is common when deriving such inequalities --- see \cite{askey}, for example).
For $j, k$ nonnegative integers, we will set
\[
p_{j, k}(x) = \frac{\geg{j}{1+\theta k}(x)}{\geg{j}{1 + \theta k}(1)}.
\]
\begin{definition}
We will say that a polynomial $p$ is {\em \fit{$\beta$}} if 
\begin{enumerate}
\item $p(x) < 0$ in $(\beta, 1)$, and
\item $p(x) > 0$ in $(1, \infty)$. 
\end{enumerate}
It is easy to check that \fit{$\beta$}ness is closed under nonnegative linear combinations (a fact that will be used in Lemma~\ref{lem:thetaorthogonal-fit}) and that Remark~\ref{rem:gammatheta} implies that any polynomial which is  \fit{$\Gamma^d_{\theta_1}$} is also \fit{$\Gamma^d_{\theta_2}$} whenever $\theta_2 < \theta_1$.
\end{definition}
The following lemma reduces the monotonicity statement we are interested in to one regarding certain polynomials being \fit{$\beta$}.

\begin{lemma}\label{lem: monotonictygegen}
The following statements are equivalent:
\begin{enumerate}
\item 
For any fixed $x > \Gamma^d_\theta$, we have 
\[
\cauchy{\geg{d}{1+\theta d}}{x} \leq \cauchy{\geg{d+1}{1+\theta (d+1)}}{x}.
\]
\item The polynomial
\begin{equation}\label{eq:qpoly}
\Delta_d(x) = p_{d+1, d+1}(x) p_{d-1, d}(x) - p_{d,d}(x) p_{d, d+1} (x)
\end{equation}
is \fit{$\Gamma^d_\theta$}.
\end{enumerate}
\end{lemma}
\begin{proof}
We start by rewriting the first statement as
\[
\frac{1}{d}\frac{\deriv{x}\geg{d}{1+\theta d}(x)}{\geg{d}{1+\theta d}(x)} 
< \frac{1}{d+1}\frac{\deriv{x}\geg{d+1}{1+\theta (d+1)}(x) }{\geg{d+1}{1+\theta (d+1)}(x)}.
\]
and notice that the third identity in Lemma~\ref{lem:identities} implies that
\[
\deriv{x} \geg{d}{1+\theta d}(x)= -dx\geg{d}{1+\theta d}(x) +\big(d-1+2(1+d\theta)\big)\geg{d-1}{1+\theta d}(x)
\]
and
\[
\deriv{x} \geg{d+1}{1+\theta (d+1)}(x)= -(d+1)x\geg{d+1}{1+\theta (d+1)}(x)+\bigg(d+2\big(1+(d+1)\theta\big)\bigg) \geg{d}{1+\theta (d+1)}(x).
\]

Plugging these in and canceling the like terms, we get the equivalent statement
\[
\frac{(d + 2 d \theta + 1)}{d(1-x^2)}
\frac{\geg{d-1}{1+\theta d}(x)}{\geg{d}{1+\theta d}(x)} 
< \frac{(d + 2 d \theta + 2 \theta + 2)}{(d+1)(1-x^2)}
\frac{\geg{d}{1+\theta (d+1)}(x)}{\geg{d+1}{1+\theta (d+1)}(x)}.
\]
Plugging in the second identity in Lemma~\ref{lem:identities} and simplifying, we see that this is equivalent to having $\Delta_d(x) < 0$ on the interval $( \Gamma^d_{\theta}, 1)$ and $\Delta_d(x) > 0$ on the interval $(1, \infty)$, which is the definition of being \fit{$\Gamma^d_\theta$}.
\end{proof}

%

Before proving anything with the polynomials $p_{j, k}$, it will be worthwhile to note the translation of the first identity in Lemma~\ref{lem:identities} into these polynomials:
\begin{equation}\label{eq:conexp}
x p_{j, k} = \lambda_{j, k} p_{j+1, k} + (1-\lambda_{j, k}) p_{j-1, k}
\end{equation}
where (for our fixed $\theta$)
\[
0 \leq \lambda_{j, k} = \frac{j + 2k \theta +2}{2j+2k\theta +2} \leq 1.
\]
We will use the following fact (that comes from direct calculation):
\begin{equation}\label{eq:diff}
\lambda_{j,k} - \lambda_{j-1, k-1} = \frac{(j - k) \theta - 1}{2(1 + j + k \theta)(j + (k-1)\theta)}.
\end{equation}

\begin{lemma}\label{lem:thetaorthogonal-fit}
The polynomial
\begin{equation}
\Delta_{j, d}(x) = 
p_{j+1, d+1}(x) p_{j-1, d}(x) - p_{j, d}(x) p_{j, d+1} (x) 
\end{equation}
is \fit{$\Gamma^d_\theta$} for all $1 \leq j \leq d$.
\end{lemma}
\begin{proof}
The fact that $x = 1$ is a root follows from the definition of $p_{j, k}$, and so it would suffice to show that $\Delta_{j, d}$ has at most one root in the desired interval, a fact that we will prove by induction (on $j$).
The base case can be computed explicitly:
\[
\Delta_{1, d} = \frac{x^2-1}{3 + 2(d+1) \theta}
\]
which is clearly \fit{$\Gamma^d_\theta$}.
For the inductive step, let 
$s = \lambda_{j, d+1}$ and $t = \lambda_{j-1, d}$.
By (\ref{eq:conexp}), we have 
\[
x p_{j, d+1} = s p_{j+1, d+1} + (1-s)p_{j-1, d+1}
\AND
x p_{j-1, d} = t p_{j, d} + (1-t)p_{j-2, d}
\]
and so
\[
t p_{j, d+1}p_{j, d} + (1-t)p_{j, d+1}p_{j-2, d} 
= s p_{j-1, d}p_{j+1, d+1} + (1-s)p_{j-1, d}p_{j-1, d+1}.
\]
Plugging in 
\[
p_{j+1, d+1}p_{j-1, d} 
= \Delta_{j, d} + p_{j, d}p_{j, d+1}
\AND
p_{j, d+1}p_{j-2, d} 
= \Delta_{j-1, d} + p_{j-1, d}p_{j-1, d+1}
\]
then gives 
\[
s \Delta_{j,d} = (1-t)\Delta_{j-1, d} + (t - s)\big(p_{j, d+1}p_{j, d} - p_{j-1, d}p_{j-1, d+1}\big)
\]
The induction hypothesis gives that $\Delta_{j-1, d}$ is \fit{$\Gamma^d_\theta$} and (\ref{eq:diff}) shows that $t - s > 0$.
Since $\theta$-orthogonal-fitness is closed under nonnegative combinations, it would then suffice to show that the polynomial $p_{j, d+1}p_{j, d} - p_{j-1, d+1}p_{j-1, d}$ is  \fit{$\Gamma^d_\theta$} for any $j \leq d$.
Again, $x = 1$ is obviously a root, and so the theorem would follow by showing that there is at most one real root in the desired interval.

However, this follows easily from interlacing properties: it is well known that consecutive orthogonal polynomial have interlacing roots, and one can show that a polynomial $p$ interlaces a polynomial $q$ if and only if there exist nonnegative constants $\lambda_i$ 
such that
\[
\frac{p}{q} 
= \sum_{i=1}^d \frac{\lambda_i}{x - r_i}
\AND
\]
where the $r_i$ are the roots of $q$ (see, for example, \cite{Wagner}).
In particular, this ratio is nonnegative and monotone decreasing at any $x > \maxroot{q}$.
Hence $p_{j-1, d}$ interlaces $p_{j, d}$ and $p_{j-1, d+1}$ interlaces $p_{j, d+1}$ and the product
\[
\frac{p_{j-1, d}}{p_{j, d}} \frac{p_{j-1, d+1}}{p_{j, d+1}}
\]
is nonnegative and monotone decreasing for $x > \maxroot{p_{j, d}}$ (recall that the monotonicity mentioned in Remark~\ref{rem:gammatheta} implies that $p_{d, k}$ has the largest root of these polynomials).
Hence
\[
\big(p_{j, d+1}p_{j, d} - p_{j-1, d}p_{j-1, d+1}\big)
= p_{j, d+1}p_{j, d}\left(1 - \frac{p_{d-1, k}}{p_{j, d}} \frac{p_{j-1, d+1}}{p_{j, d+1}} \right)
\]
can have at most one solution in the interval $(\maxroot{p_{j, d}}, \infty)$.
But for $j \leq d$, we have 
\[
\maxroot{p_{j, d}} = \maxroot{\geg{j}{1+\theta d}(x)} \leq \maxroot{\geg{d}{1+\theta d}(x)}  \leq \Gamma^{d}_\theta
\]
implying that $p_{j, d+1}p_{j, d} - p_{j-1, d}p_{j-1, d+1}$ has a single root in $(\Gamma^d_\theta, \infty)$ and proving the theorem.
\end{proof}

Note that, even though we were forced to consider the possibility that $\gamma^d_\theta > \gamma^{d+1}_\theta$ in the proof of Lemma~\ref{lem:thetaorthogonal-fit}, one direct implication of the lemma is that such a scenario is impossible.

\begin{corollary}\label{cor:prepreans}
For all $\theta > 0, \nn \geq 0$, and $\dd \geq 1$, the sequence
\[
 \left\{ \maxroot{\geg{\dd}{\dd \theta +1}} \right\}_d
\]
is monotone increasing, and
\[
 \lim_{d \to \infty} \maxroot{\geg{\dd}{\dd \theta +1}} = \gamma_{\theta} 
\]
where $\gamma_\theta = \frac{\sqrt{2 \theta + 1}} {\theta+1}$ (as in Lemma~\ref{laforgia}).
\end{corollary}
\begin{proof}

By Theorem~\ref{thm:asymptgegen}, we know that
\begin{equation}\label{eq:ineqthing}
\lim_{d \to \infty} \maxroot{\geg{d}{1 + d \theta}(x)} \leq \frac{\sqrt{2 \theta + 1}} {\theta+1} = \gamma_\theta
\end{equation}
and the combination of Lemma~\ref{lem: monotonictygegen} and Lemma~\ref{lem:thetaorthogonal-fit} show that this convergence is monotone increasing.
On the other hand, Lemma~\ref{laforgia} shows that the asymptotic root distribution is dense in $[ -\gamma_\theta, \gamma_\theta ]$, and so the inequality in (\ref{eq:ineqthing}) must be an equality.
\end{proof}

For our purposes, we will need a similar statement about the Cauchy transform:

\begin{corollary}\label{cor:preans}
For $n \geq 0$ and $d \geq 1$, we have
\begin{enumerate}
\item $\maxroot{\geg{\dd}{\nn+1}} \leq \gamma_{\nn/\dd}$
\item for all $x > \gamma_{\nn/\dd}$, we have
\[
\cauchy{\geg{d}{n + 1}}{x} 
\leq \frac{-\nn x + \sqrt{(\nn+\dd)^2(x^2-1) + \nn^2}}{\dd(x^2-1)}
\]
\end{enumerate}
where $\gamma_{\nn/\dd} = \sqrt{1 - \frac{\nn^2}{(\nn+\dd)^2}}$.
\end{corollary}
\begin{proof}
The first claim follows directly from Corollary~\ref{cor:prepreans} with $\theta = \nn / \dd$.
For the second claim, Corollary~\ref{cor:asymptcauchygegen} shows that
\[
\lim_{d \to \infty} \cauchy{\geg{d}{1 + \theta d}}{x} 
=\frac{-\theta x+(1+\theta)\sqrt{x^2-\gamma_\theta^2}}{(x^2-1)}. 
\]
and the combination of Lemma~\ref{lem: monotonictygegen} and Lemma~\ref{lem:thetaorthogonal-fit} show that this convergence is monotone increasing.
Setting $\theta = n/d$ and simplifying then gives the corollary.
\end{proof}

Note that while the form of Corollary~\ref{cor:preans}.2 is the more popular one in the literature, it has a downside when appearing in inequalities (the false appearance of a sign change at $x = 1$).
It is easy to check by cross multiplication that an equivalent way to write this inequality is
\begin{equation}\label{eq:altpreans}
\cauchy{\geg{d}{1 + \nn}}{x} 
\leq 
\frac{2\nn + \dd}{\nn x + \sqrt{ (\nn + \dd)^2(x^2-1) + \nn^2}}
\end{equation}
and this will be the form we use in our proof of Theorem~\ref{cor:preans}.

\begin{theorem}\label{thm:ans}
Consider the bivariate polynomial 
\[
f(x, y) = \dd (x^2 -1)y^2 + 2 \nn x y - (2\nn + d)
\]
and assume that for a given $s > \gamma_{\nn/\dd}$ and $t > 0$ that $f(s, t) \geq 0$.
Then 
\[
\cauchy{C_{\dd}^{\nn+1}}{s} \leq t.
\]
\end{theorem}
\begin{proof}
Fix $u > \gamma_{\nn/\dd}$ and let $G = \cauchy{\geg{d}{n + 1}}{u}$.
By Corollary~\ref{cor:preans} and (\ref{eq:altpreans}), we have 
\[
\frac{1}{G} \geq 
\frac{\nn u + \sqrt{ (\nn + \dd)^2(u^2-1) + \nn^2}}{2\nn + \dd}
\Longrightarrow
\frac{2\nn + \dd}{G} - \nn u \geq 
\sqrt{ (\nn + \dd)^2(u^2-1) + \nn^2}.
\]
Since both sides are positive, we can square them, to get
\[
\frac{(2\nn + \dd)^2}{G^2} -  \frac{2\nn t(2\nn + \dd)}{G} + \nn^2 u^2 \geq 
(\nn + \dd)^2(u^2-1) + \nn^2.
\]
Using the fact that $G > 0$, we can then simplify to get 
\[
\dd(u^2-1) G^2 + 2 \nn u G - (2 \nn + d) \leq 0.
\]
Hence for any $u > \gamma_{\nn/\dd}$, we have that
\begin{equation}\label{eq:smallerthan0}
f\left(u, \cauchy{\geg{d}{n + 1}}{u} \right) \leq 0.
\end{equation}

Now let $s > \gamma_{\nn/\dd}$ and $t > 0$ satisfy $f(s, t) \geq 0$.
It is well known that the range of the Cauchy transform is the positive reals, so there exists some $r$ for which $t = \cauchy{\geg{d}{n + 1}}{r}$.
It then suffices to show that $r \leq s$, since (due to the fact that the Cauchy transform is decreasing) that would imply
\[
\cauchy{\geg{d}{n + 1}}{s} \leq \cauchy{\geg{d}{n + 1}}{r} = t.
\]
as required.

To see that $r \leq s$, we consider two cases.
If $r \leq \gamma_{\nn/\dd}$, then $r \leq s$ trivially.
Otherwise, we have $r > \gamma_{\nn/\dd}$, so by (\ref{eq:smallerthan0}), we must have $f(r, t) \leq 0 \leq f(s, t)$.
However it is easy to check that 
\begin{equation}\label{eq:increasex}
\frac{\partial}{\partial x} f(x, y) = 2 \dd x y^2 + 2 \nn y > 0
\end{equation}
whenever $x, y \geq 0$, implying $r \leq s$ in this case as well.
\end{proof}

\section{Illustrative Examples}\label{sec:examples}

\newcommand{\aplus}{~\tilde{+}~}
\newcommand{\ttt}{\dd}
\newcommand{\class}[1]{[ #1] }

In this section we hope to give some intuition as to how one can view the 
$\Wtrans{\nn}{\uu}{}$ operator and (in particular) the role that the parameter 
$\uu$ plays.
We then give computational examples that show the relative accuracy of 
Theorem~\ref{thm:main}.
All plots and computations in this section were done using Mathematica 12.

\subsection{Singular values of rectangular matrices}

Before discussing polynomials directly, it will be informative to recall the 
connection with the singular values of matrices that was mentioned  
Section~\ref{sec:ffp}.
That is, let $A$ and $B$ be $(\dd + \kk) \times \dd$ matrices with 
\[
p_A(x) = \mydet{x \ident - AA^*}
\AND
p_B(x) = \mydet{x \ident - BB^*}.
\]
so that the roots of $\Strans p$ are
\[
\{ \pm \sigma_1(A), \dots, \pm \sigma_d(A) \}
\]
(and similar for $B$ and $q$).
One (important) consequence of Theorem~\ref{thm:rr} is that there exists a 
$(\dd + \kk) \times \dd$ matrix $C$ for which 
\[
[p_A \mysum{\dd}{\kk} p_B](x) = \mydet{x \ident - CC^*}.
\]
Thus if we define $\class{A}$ to be the class of matrices which have the same 
dimensions 
and same singular values as $A$, then the rectangular additive convolution can 
be viewed as a binary operation on these classes.
Coupled with the observation made in Section~\ref{sec:ffp} that we have the 
explicit 
formula
\begin{equation}\label{eq:conv}
[p_A \mysum{\dd}{\kk} p_B](x) = \int_{Q, R} \mydet{ x \ident - (A + Q B R)(A + 
Q B 
R)^*} dQ dR
\end{equation}
when $dQ$ and $dR$ are Haar-uniform random unitary matrices of the appropriate 
size, one could feel justified in writing this binary relation as
\[
\class{A} \aplus \class{B} = \class{C}.
\]

An obvious question, then, is to what extent $\class{A} \aplus \class{B}$ and 
$\class{A + B}$ 
have 
similar behavior.
In general, they cannot be the same because the singular values of $A + B$ will 
depend on the singular vectors of $A$ and $B$, whereas $\class{A} \aplus 
\class{B}$ is 
independent of the singular vectors.
However, if we were to interpret $\aplus$ as some sort of ``unitarily 
invariant'' way of ``adding'' two matrices together, then 
Theorem~\ref{thm:main} would be a statement about the behavior of the largest 
singular value under this operation.

The first relevant observation in this direction is that 
Theorem~\ref{thm:main} simplifies greatly when $\alpha = 0$, in that 
\[
\Theta^{\nn}_{0}(p) = \maxroot{\Strans p}.
\]
Hence in this case, Theorem~\ref{thm:main} reduces to 
\[
\maxroot{\Strans[p \mysum{\dd}{\kk} q]} \leq \maxroot{\Strans p} + \maxroot{ 
\Strans q}
\]
which, when written in the matrix context, becomes the inequality
\begin{equation}\label{eq:triangle}
\sigma_{\max}(\class{A} \aplus \class{B}) \leq \sigma_{\max}(\class{A}) + 
\sigma_{\max}(\class{B}).
\end{equation}
That is, $\aplus$ satisfies a triangle inequality similar to normal matrix 
addition (and in fact this triangle inequality can be derived from the normal 
one using \eqref{eq:conv}).

The shortcoming of \eqref{eq:triangle} is that it neglects much of the 
information 
we have concerning the singular values of the original matrices.
If we are to accept the interpretation that $\aplus$ is some sort of a 
``unitarily invariant'' version of addition, then we should suspect that the 
true answer will depend on all of the singular values of $A$ and $B$.
This is the case, and the purpose of $\Wtrans{\nn}{\uu}{}$ is to try 
to obtain better bounds by accessing this extra information.

\subsection{The role of $\Wtrans{\nn}{\uu}{}$}

For the purpose of illustration, let us consider the matrix classes $\class{A}$ 
and $\class{B}$ with: 
\[
\sigma_\ttt(\class{A}) = \sigma_{\ttt-1}(\class{A}) = \dots = 
\sigma_2(\class{A}) = 1
\AND
\sigma_1(\class{A}) = 2.
\]
versus 
\[
\sigma_\ttt(\class{B}) = 1
\AND
\sigma_{\ttt-1}(\class{B}) = \sigma_{\ttt-2}(\class{B}) = \dots = 
\sigma_1(\class{B}) = 2
\]
It should be clear that the triangle inequality \eqref{eq:triangle} would treat 
all of the sums involving $A$ and $B$ similarly:
\[
\sigma_{\max}(\class{A} \aplus \class{A}) \leq 4 
\AND
\sigma_{\max}(\class{A} \aplus \class{B}) \leq 4.
\AND
\sigma_{\max}(\class{B} \aplus \class{B}) \leq 4.
\]
However, the interpretation of $\aplus$ as a ``unitarily 
invariant'' version of addition suggests that $\sigma_{\max}(\class{A} \aplus 
\class{A})$ 
should be significantly larger than $\sigma_{\max}(\class{B} \aplus 
\class{B})$, since 
many of the 
rotations $Q, 
R$ in \eqref{eq:conv} will result in 
\[
\sigma_{\max}(A + Q A R) = 4
\AND
\sigma_{\max}(B + Q B R) = 3.
\]
The goal then is to improve this bound by taking into account the positions of 
the other singular values (of both matrices).
However, it is important for applications that we incorporate this information 
in a way that can be easily iterated, as we will often find ourselves wanting 
to apply knowledge we have gained about $\sigma_{\max}(\class{A} \aplus 
\class{B})$ to gain knowledge about $\sigma_{\max}(\class{A} \aplus \class{B} 
\aplus \class{C})$ (for example).
The upshot of Theorem~\ref{thm:main} is that we can accomplish this using the 
$\Wtrans{\nn}{\uu}{}$ transformation.

The intuition behind the action of $\Wtrans{\nn}{\uu}{}$ that we find the most 
illuminating is inspired by a model from \cite{nikhil_thesis}.
If one were to view the roots of $p(x)$ as particles, then the application of 
$\Wtrans{\nn}{\uu}{}$ can be seen as a gust of wind (with total strength 
parametrized by $\alpha$) pushing the particles forward.
In theory, one would expect any such wind to divide its force evenly between 
the particles (pushing them the same amount).
However the operator $\Wtrans{\nn}{\uu}{}$ has the property that the positive 
roots of $p$ and those of $\Wtrans{\nn}{\uu}{p}$ will always {\em 
interlace}\footnote{We will not prove this here, but it is a direct consequence 
of Theorem~\ref{thm:rr} --- see \cite{BB2}.}.
That is, if 
\[
s_1 \geq s_2 \geq \dots \geq s_\dd
\AND
s'_1 \geq s'_2 \geq \dots \geq s'_\dd
\]
are the roots of $p$ and $\Wtrans{\nn}{\uu}{p}$ (respectively), then one will 
always have the inequality
\[
s_{i-1} \geq s_{i}' \geq s_{i}.
\]
As a result, the force of the wind treats the particles like ``billiard balls'' 
with polynomials like $p_B$ (that have many roots close to the largest root) 
transferring much of the energy onto the largest root.
In comparison, $p_A$ will see a lot of its energy transferred onto its  
second-largest root, but as long as $\uu$ is not too big, there will not be 
much energy transferred on to the largest root.

Hence the distance that the operator $\Wtrans{\nn}{\uu}{}$ pushes the root of a 
polynomial, when viewed as a function of $\uu$, encodes different levels of 
information about the locations of the other roots.
Values of $\uu$ that are too large will result in most of the energy being 
transferred onto the largest root (independent of the initial configuration).
Similarly, small values of $\uu$ will result in very little of the energy being 
transferred on to the largest root (also independent of the initial 
configuration).
Intuitively, the optimal $\uu$ will be one that allows as much initial force as 
possible while still avoiding too much transfer onto the largest particle.

\subsection{Example computations}

Let use consider the examples from the previous section in the case where $A$ 
and $B$ are $6 \times 9$ matrices (so $d = 6, n = 3$). 
Hence
\[
p_A(x) = (x-1)^5(x-4)
\AND
p_B(x) = (x-1)(x-4)^5.
\]
and the resulting convolutions are (where we have replaced $\mysum{6}{3}$ with 
$\boxplus$ to aid readability):
\begin{align*}
[p_A \mysum{}{} p_A](x) 
&= 
x^6 - 18 x^5 + 120 x^4 - 380 x^3 + 600 x^2 - 442 x + \frac{2431}{21} \\
[p_A \mysum{}{} p_B](x) 
&= 
x^6 - 30 x^5 + 350 x^4 - \frac{6050}{3} x^3 + 5975 x^2 - 8450 x + 
\frac{90080}{21} \\
[p_B \mysum{}{} p_B](x) 
&= 
x^6 - 42 x^5 + \frac{2060}{3} x^4 - 5520 x^3 + 22600 x^2 - \frac{130624}{3} x + 
\frac{622336}{21}.
\end{align*}
Figure~\ref{fig:alphas} shows two 
functions of $\alpha$, plotted for each of the convolutions above: the true 
value of 
$\maxroot{\Wtrans{3}{\uu}{q}}$ (in yellow) plotted against
the bound given by Theorem~\ref{thm:main} (in blue). 
In particular, the blue curve in each case will have value $4$ when $\alpha = 
0$ (as dictated by the triangle inequality), while the value we are trying to 
bound will be the value of the yellow curve at $\alpha = 0$.
Unfortunately, we were forced to cut off the values at $\alpha = 0$ in the 
graphs due to precision issues, so we have listed the relevant values in 
Table~\ref{table}.

\begin{figure}[!h]
    \centering

    \begin{minipage}{0.25\textwidth}
        \centering
        \includegraphics[width=\linewidth, 
        height=0.15\textheight]{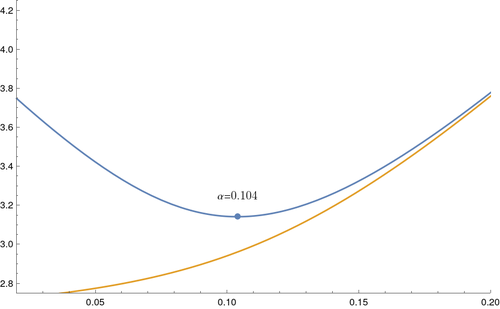}
$q = p_A \mysum{}{} p_A$
    \end{minipage}
\qquad
    \begin{minipage}{0.25\textwidth}
        \centering
        \includegraphics[width=\linewidth, 
        height=0.15\textheight]{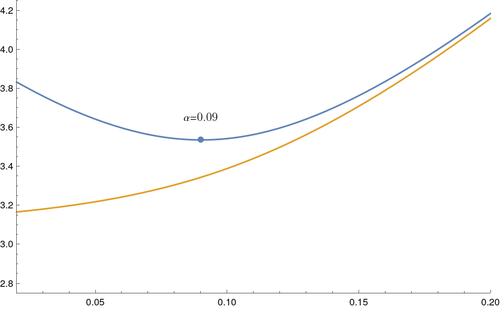}
$q = p_A \mysum{}{} p_B$

    \end{minipage}
\qquad
    \begin{minipage}{0.25\textwidth}
        \centering
        \includegraphics[width=\linewidth, 
        height=0.15\textheight]{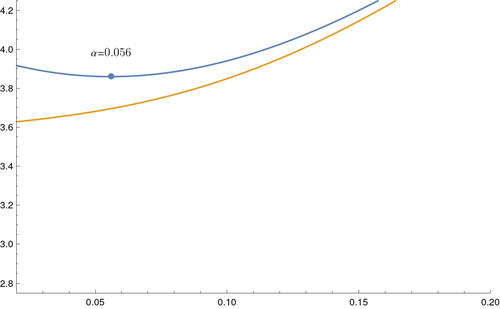}
$q = p_B \mysum{}{} p_B$
    \end{minipage}

\caption{Optimal values of $\alpha$ for three convolutions involving $p_A$ and 
$p_B$.
The blue curve shows the bound provided by 
Theorem~\ref{thm:main} as a function of $\alpha$. 
The yellow curve shows $\maxroot{\Wtrans{3}{\uu}{q}}$ for the convolution 
polynomial $q$.
}
\label{fig:alphas}
\end{figure}
\begin{table}[!h]
\def\arraystretch{1.4}
\[
\begin{array}{c|c|c|c|c}
\text{Convolution $q(x)$} &  \text{Best 
$\uu_0$ in Thm.~\ref{thm:main}} &  \text{Resulting bound} & 
\maxroot{q(x^2)}   & 
\text{\eqref{eq:triangle} bound} \\
\hline
p_A \mysum{}{} p_A &  0.104 & 3.143 & 2.702 & 4 \\
p_A \mysum{}{} p_B &  0.090 & 3.537 & 3.144 & 4 \\
p_B \mysum{}{} p_B &  0.056 & 3.861 & 3.606 & 4
\end{array}
\]
\caption{The best $\alpha$ when Theorem~\ref{thm:main} is 
applied to each convolution, the resulting bound, and the true maximum (the 
yellow curve in Figure~\ref{fig:alphas} at $\alpha = 0$).}
\label{table}
\end{table}

\section{Further Research}\label{sec:conclusion}

One possible direction of research comes from an observation in \cite{BG} concerning a degeneration in the rectangular convolution in the case $\nn = 0$ ($\lambda = 1$ in \cite{BG}) that allows one to compute singular value distributions of square matrices using the additive convolution.
In particular, if the asymptotic singular laws of two independent (left and right unitarily invariant) matrices converge to $\mu_1$ and $\mu_2$, then the asymptotic singular law of their sum is the unique probability measure on $[0, \infty)$ which, when symmetrized, is equal to the free convolution of the symmetrizations of $\mu_1$ and $\mu_2$.
In terms of polynomial convolutions, this would mean that 
\[
\Strans{[p \mysum{\dd}{0} q]} 
= [\Strans{p} \mysum{2\dd}{} \Strans{q}]
\]
which is not true (in general).
It is true that they share the same bounds; by Theorem~\ref{thm:symmetric},
\[
\maxroot{U_\alpha[\Strans{p} \mysum{2\dd}{} \Strans{q}]} \leq \maxroot{U_\alpha \Strans{p}} + \maxroot{U_\alpha \Strans{q}} - 2 \alpha \dd
\]
and by Theorem~\ref{thm:asymmetric},
\[
\maxroot{U_\alpha \Strans{[p \mysum{\dd}{0} q]}} \leq \maxroot{U_\alpha \Strans{p}} + \maxroot{U_\alpha \Strans{q}} - 2 \alpha \dd.
\]
One might then ask whether there is an inequality between the two, and we conjecture that, in fact, there is:
\begin{conjecture}\label{conj:2dvsd2}
For all $p, q$ in $\rrposle{d}$, we have
\[
\maxroot{U_\alpha \Strans{[p \mysum{\dd}{0} q]} ]} \leq \maxroot{U_\alpha[\Strans{p} \mysum{2\dd}{} \Strans{q}]}.
\]
\end{conjecture}

It is not hard to show that Conjecture~\ref{conj:2dvsd2} is true for basic 
polynomials in $\rrpos{d}$.
Recall from Corollary~\ref{cor:rectToGeg} that for basic polynomials $p(x) = (x 
- \lambda)^d$ and $q(x) = (x-\mu)^d$ that
\[
\Strans{[p \mysum{d}{0} q]}(y) = (\lambda \mu)^{d/2} C_{d}^{1} \left(\frac{y^2 - (\lambda +\mu)}{2 \sqrt{\lambda \mu }} \right).
\]
Using a similar method, one can show that for these polynomials
\[
[\Strans{p} \mysum{2d}{} \Strans{q}](y) = \frac{(16\lambda \mu)^{d/2}}{\binom{2d}{d}} C_{d}^{1/2} \left(\frac{y^2 - (\lambda +\mu)}{2 \sqrt{\lambda \mu }} \right).
\]
As was noted in Remark~\ref{rem:gammatheta}, the positive roots of 
$C_{d}^{\alpha}(x)$ are decreasing with $\alpha$, which will result in the 
roots of $[\Strans{p} \mysum{2d}{} \Strans{q}]$ majorizing the roots of 
$\Strans{[p \mysum{d}{0} q]}(y)$.
Since the function $1/(x-t)$ is convex for $x > t$, the Cauchy transform is 
Schur convex on that range, and this implies Conjecture~\ref{conj:2dvsd2} (for 
this particular $p$ and $q$).
In particular, it may be possible to prove Conjecture~\ref{conj:2dvsd2} using an inductive method like the one developed in Section~\ref{sec:inductions}.

Later versions of \cite{convolutions4} use a similar argument on basic 
polynomials to prove the ``base case'' for Theorem~\ref{thm:asymmetric} from 
Theorem~\ref{thm:symmetric} and, if true in general, 
Conjecture~\ref{conj:2dvsd2} would imply Theorem~\ref{thm:asymmetric} 
directly from Theorem~\ref{thm:symmetric}.
It would be interesting to know whether Theorem~\ref{thm:main} can be proved as 
a corollary of  Theorem~\ref{thm:symmetric}, as this would imply new identities 
in free probability.
However, this seems unlikely; as was noted in Section~\ref{sec:previous}, the only reason Theorem~\ref{thm:asymmetric} can be stated using the $U_\alpha$ operator is due to a degeneration at $n = 0$ that does not hold for any other $n$ and there is no reason to believe that $U_\alpha$ and $\Theta_\alpha^n$ can be compared in an asymptotically tight way for general $n$. 

Lastly, it would be interesting to determine whether Theorem~\ref{thm:main} is a special case of some submodularity inequality similar to Theorem~\ref{thm:symmetric} and Theorem~\ref{thm:submodular}.

\section*{Appendix}\label{sec:appendix}

\newcommand{\mtn}{\mu_\theta^\nn}

In this appendix, we give a brief sketch of the computation leading from 
Theorem~\ref{thm:asymptgegen} to Corollary~\ref{cor:asymptcauchygegen}.
We have seen such a calculation referred to as ``standard'' in various places 
in the literature, but wanted to give some indication as to how the proof goes 
for those who, like the authors, might be less familiar with such computations.

\begin{lemma}\label{lem:asymptcauchygegen}
For $\theta > 0$ and $\nn > 0$, let $\mtn$ be a sequence of compact distributions 
for which 
\[
\mtn(y) \xrightarrow{\nn \to \infty} \mu_\theta(y) = \frac{1 + \theta}{\pi} 
\frac{\sqrt{\gamma_\theta^2 - y^2}}{1-y^2} \one_{[-\gamma_\theta, 
\gamma_\theta]}(y)
\]
where $\gamma_\theta = \frac{\sqrt{2 \theta + 1}} {\theta+1}$.
Then for all $x > \gamma_\theta$, we have
\[
\cauchy{\mtn}{y} \xrightarrow{\nn \to \infty}
\cauchy{\mu_\theta}{x} =\frac{-\theta x+(1+\theta)\sqrt{x^2-\gamma_\theta^2}}{(x^2-1)}. 
\]
\end{lemma}

\begin{proof}
Since all of the distributions are compact, we can interchange the limit with the integral defining the Cauchy transform.
That is, 
\[
\lim_{\nn \to \infty} \cauchy{\mtn}{x}
= \int \lim_{\nn \to \infty} \frac{\mtn(z)} {x-z}\d{z}
=\frac{1+\theta}{\pi} \int_{-\gamma_\theta}^{\gamma_\theta} 
\frac{\sqrt{\gamma_\theta^2-z^2}}{(x-z)(1-z^2)}~\d{z}.
\]
We then make a change of variable from $z$ to $u$, using the Euler substitution: 
\[
z
=\frac{-2\gamma_\theta u}{1+u^2}
\qquad \text{so that} \qquad 
\d{z} = 2 \gamma_\theta \frac{u^2-1}{(u^2+1)^2}~\d{u},
\]
yielding
\begin{equation}\label{eq:substituted}
\lim_{\nn \to \infty} \cauchy{\mtn}{x}
=\frac{1+\theta}{\pi}\int_{-1}^{1} \frac{ 2\gamma_\theta^2 
(1-u^2)^2}{(xu^2+2\gamma_\theta u +x)\big( 
(1+u^2)^2-4\gamma_\theta^2u^2\big)}~\d{u}.
\end{equation}
We can then rewrite the integrand using partial fractions:
\[
\frac{ 2\gamma_\theta^2 (1-u^2)^2}{(xu^2+2\gamma_\theta u +x)
\big( (1+u^2)^2-4\gamma_\theta^2u^2 \big)}
= 2\frac{x^2-\gamma_\theta^2}{(x^2-1)x} g(\frac{1}{x}) 
+\frac{\gamma_\theta^2-1}{(x+1)} g(-1)
+ \frac{\gamma_\theta^2-1}{(x-1)} g(1)
\]
where we have defined
\[
g(s) 
:= \frac{1}{u^2 + 2 s u \gamma_\theta +1}
=
\frac{1}{(u - s \gamma_\theta)^2 + 1 - s^2 \gamma_\theta^2}.
\]
For $|s \gamma_\theta| < 1$, these integrals can be computed explicitly using the trigonometric substitution $u =  \tan(\theta) \sqrt{1 - \gamma_{\theta}^2 s^2}$.
This gives 
\[
\int_{-1}^1 g(s)~\d{u} 
= \frac{1}{\sqrt{1 - \gamma_\theta^2s^2}} \left( \arctan\left( \sqrt{\frac{1 - \gamma_\theta s}{1 + \gamma_\theta s}} \right) + \arctan\left( \sqrt{\frac{1 + \gamma_\theta s}{1 - \gamma_\theta s}} \right) \right) 
= \frac{\pi}{2\sqrt{1 - \gamma_\theta^2s^2}}
\]
since $\arctan(z) + \arctan(1/z) = \pi/2$ for all $z$.
The result follows by plugging these into (\ref{eq:substituted}) and simplifying.
\end{proof}

\end{document}